\documentclass[11pt]{article}
\usepackage[utf8]{inputenc}
\usepackage[english]{babel}
 
\usepackage[dvipsnames]{xcolor}

\usepackage{array}
\usepackage{mleftright}
\usepackage[nottoc]{tocbibind}
\usepackage{amsmath}
\usepackage{cases}
\usepackage{amsfonts}
\usepackage{amssymb}
\usepackage{array}
\usepackage{mathtools}
\usepackage{tikz}
\usepackage{tikz-cd}
\usepackage{ragged2e}
\usepackage{amsthm}
\usepackage{tabularx}
\usepackage{centernot}
\usepackage{faktor}
\usepackage{rotating}
\usepackage{commath}
\usepackage{comment}
\usepackage{mathabx}
\usepackage{fge}
\usepackage[parfill]{parskip}
\usepackage[shortlabels]{enumitem}
\usepackage{xifthen}
\usepackage{etoolbox}
\usepackage{verbatim}
\usepackage{transparent}
\usepackage{tikzducks}
\usepackage{ytableau}
\usepackage{subfig}
\usepackage{slashed}
\usepackage{xfrac}
\usepackage[title]{appendix}
\usepackage{multirow}
\usepackage{textcomp}
\usepackage[hyphens]{url}

\usepackage[margin=1.15in]{geometry}

\usepackage{pdflscape}
\usepackage{afterpage}
\usepackage{capt-of}

\usepackage{lipsum}
\usepackage{hyperref}

\hypersetup{
    colorlinks=true,
    linkcolor=BrickRed,
    filecolor=Apricot,      
    urlcolor=Rhodamine,
    citecolor=ForestGreen,
    anchorcolor=Fuchsia,
}

\let\norm\undefined 
\DeclarePairedDelimiter\norm{\lVert}{\rVert}

\DeclareFontFamily{U}{MnSymbolC}{}
\DeclareSymbolFont{MnSyC}{U}{MnSymbolC}{m}{n}
\DeclareFontShape{U}{MnSymbolC}{m}{n}{
    <-6>  MnSymbolC5
   <6-7>  MnSymbolC6
   <7-8>  MnSymbolC7
   <8-9>  MnSymbolC8
   <9-10> MnSymbolC9
  <10-12> MnSymbolC10
  <12->   MnSymbolC12}{}
\DeclareMathSymbol{\hook}{\mathbin}{MnSyC}{'270}

\newcommand{\bb}[1]{
    \mathbb{#1}
}
\renewcommand{\cal}[1]{
    \mathcal{#1}
}

\newcommand{\eps}{
    \epsilon
}
\newcommand{\rarr}{
    \rightarrow
}

\renewcommand{\:}{
    \colon
}

\renewcommand{\del}[1]{
    \partial #1
}
\newcommand{\delbar}[1]{
    \overline{\partial} #1
}

\newcommand{\ie}{\textit{i.e.}
}
\newcommand{\cf}{\textit{c.f.}
}
\newcommand{\eg}{\textit{e.g.}
}

\newcommand{\rank}[1]{
    \text{\normalfont rank}\,#1
}

\renewcommand{\ker}[1]{
    \text{\normalfont ker}\,#1
}

\newcommand{\FLY}{\mathrm{FLY}}

\newcommand{\bs}{
    \char`\\
}
\newcommand{\wtilde}[1]{
    \widetilde{#1}
}
\newcommand{\what}[1]{
    \widehat{#1}
}

\newcommand{\co}{\mathrm{co}}

\newcommand{\Vol}{
    \text{\normalfont{Vol}}
}
\renewcommand{\Re}{
    \text{\normalfont{Re}}
}
\renewcommand{\Im}{
    \text{\normalfont{Im}}
}
\newcommand{\id}{
    \text{\normalfont{id}}
}

\newcommand{\diam}{
    \text{\normalfont{diam}}
}

\newtheorem{innercustomgeneric}{\customgenericname}
\providecommand{\customgenericname}{}
\newcommand{\newcustomtheorem}[2]{%
  \newenvironment{#1}[1]
  {%
   \renewcommand\customgenericname{#2}%
   \renewcommand\theinnercustomgeneric{##1}%
   \innercustomgeneric
  }
  {\endinnercustomgeneric}
}

\makeatletter
\providecommand{\subtitle}[1]{
  \apptocmd{\@title}{\par {\large #1 \par}}{}{}
}
\makeatother

\makeatletter
\def\mathcolor#1#{\@mathcolor{#1}}
\def\@mathcolor#1#2#3{%
  \protect\leavevmode
  \begingroup
    \color#1{#2}#3%
  \endgroup
}
\makeatother

\newcustomtheorem{customthm}{Theorem}
\newcustomtheorem{customlem}{Lemma}
\newcustomtheorem{customprop}{Proposition}

\theoremstyle{plain}
\newtheorem{thm}{Theorem}[section]
\newtheorem{lem}[thm]{Lemma}

\newtheorem{propn}[thm]{Proposition}
\newtheorem{cor}[thm]{Corollary}

\theoremstyle{definition}
\newtheorem{defn}[thm]{Definition}

\newtheorem{rmk}[thm]{Remark}

\theoremstyle{remark}

\title{Gromov--Hausdorff continuity of non-K\"ahler Calabi--Yau conifold transitions}
\author{Benjamin Friedman\thanks{Department of Mathematics, UBC, 1984 Mathematics Road,
    Vancouver, Canada, \href{benji@math.ubc.ca}{benji@math.ubc.ca}}\, \quad S\'ebastien Picard\thanks{Department of Mathematics, UBC, 1984 Mathematics Road,
    Vancouver, Canada, \href{spicard@math.ubc.ca}{spicard@math.ubc.ca}}\, \quad Caleb Suan\thanks{Department of Mathematics, CUHK, Lady Shaw Building,
    Shatin, Hong Kong, \href{kwsuan@math.cuhk.edu.hk}{kwsuan@math.cuhk.edu.hk}}}
\date{\today}

\begin{document}
\maketitle

\begin{abstract}
We study the geometry of Calabi--Yau conifold transitions. This deformation process is known to possibly connect a K\"ahler threefold to a non-K\"ahler threefold. We use balanced and Hermitian--Yang--Mills metrics to geometrize the conifold transition and show that the whole operation is continuous in the Gromov--Hausdorff topology.
\end{abstract}

\section{Introduction}
\label{sect-intro}

Our discussion begins with the K\"ahler Calabi--Yau threefold. Our broad goal is to understand the geometric properties of these complex manifolds as they undergo deformation. Various mechanisms for the degeneration and resolution of Calabi--Yau structures exist, and in this work we focus on the conifold transition.

A conifold transition is a process where a birational contraction of holomorphic curves is followed by a deformation of complex structure. We denote a conifold transition by
\[
\what{X} \rightarrow X_0 \rightsquigarrow X_t.
\]
In this process, holomorphic curves in $\what{X}$ are mapped to singular points in the analytic space $X_0$, and the singularities are locally modeled by $0 \in \{ \sum z_i^2 = 0 \} \subset \mathbb{C}^4$. The smoothings $X_t$ deform the complex structure of $X_0$ in a way which is locally modeled by $\{ \sum z_i^2 = t\} \subset \mathbb{C}^4$. 

As the initial threefold $\hat{X}$ is deformed into $X_t$, its Hodge numbers undergo jumps. This implies that distinct threefolds with varying topologies can be interconnected through this deformation process. It is conjectured that all K\"ahler Calabi--Yau threefolds can be linked by conifold transitions \cite{candelasrolling, reid87,green1988,Fri91}, and for an introduction to conifold transitions, we refer readers to \cite{rossi}.

The goal of this work is to identify a suitable sense in which conifold transitions are continuous, even though the Hodge numbers change discretely. This is a well-studied phenomenon in string theory, and there are various string theoretic interpretations \cite{greene95, strominger95, CdlO90,anderson2022} of the smooth interpolation of string theory through topological change of Calabi--Yau threefolds. From our perspective as differential geometers, we endow the Calabi--Yau threefolds with special Riemannian metrics and study their degenerations through conifold transitions. 

K\"ahler metrics are not suitable for this purpose, as a conifold transition may connect a projective threefold to a non-K\"ahler complex manifold. A simple example is given by letting $\what{X}$ be a smooth quintic threefold. In this case, $b_2(\what{X})=1$, and so once holomorphic curves are contracted, the resulting manifolds $X_t$ have $b_2(X_t)=0$. For a more in-depth discussion of this example, readers are directed to \cite{Fri91}.

There are nevertheless many examples where the resulting manifold $X_t$ does admit a K\"ahler structure, and in this case there exists a significant body of literature dedicated to understanding the degeneration and resolution process via K\"ahler Ricci-flat metrics. For the local model of the conifold transition, families of Calabi--Yau metrics were constructed by Candelas--de la Ossa \cite{CdlO90}. On compact K\"ahler threefolds, K\"ahler Ricci-flat metrics exist by Yau's theorem \cite{yau78}, and the challenge is to carry these metrics through a conifold transition. The work of Ruan--Zhang \cite{RZ11b}, Rong--Zhang \cite{RZ11} and J. Song \cite{Song14, Song15} give the existence of a sequence
\[
(\what{X}, \what{g}_a) \rightarrow (X_0, d_0) \leftarrow (X_t,g_t)
\]
where the metrics $\what{g}_a$ $g_t$ are smooth K\"ahler Ricci-flat metrics converging in the Gromov--Hausdorff topology. The limiting length space $(X_0,d_0)$ is the metric completion of $(X_{\rm reg},g_0)$, where $g_0$ is a singular Calabi--Yau metric constructed by Eyssidieux--Guedj--Zeriahi \cite{egz}. The metrics $\what{g}_a$ on the small resolution converge smoothly uniformly on compact sets away from the exceptional curves by work of Tosatti \cite{tosatti2009}, and the metrics $g_t$ on the smoothings converge smoothly uniformly on compact sets away from the singularities by work of Ruan-Zhang \cite{RZ11b} and Rong-Zhang \cite{RZ11}. 

 For a survey on degenerations of Calabi--Yau metrics, we refer to \cite{tosatti2018}, and for recent work on understanding Calabi--Yau metrics near isolated singularities we refer to \eg \cite{DS2,heinsun,di2022families,xinfu,chiusze} and references therein.

The current work takes initial data to be a K\"ahler threefold $\what{X}$ and investigates conifold transitions emanating from $\what{X}$ without imposing a priori assumptions on the resulting space $X_t$. This setup has the implication that $X_t$ may or may not be K\"ahler. Instead of relying on K\"ahler Ricci-flat metrics, the idea in \cite{FLY12,CPY21} is to geometrize the conifold transition by a pair of metrics:
\[
(\what{X}, \what{g}_a, \what{H}_a) \rightarrow (X_0, d_{g_0}, d_{H_0}) \leftarrow (X_t,g_t,H_t).
\]
Here $(g,H)$ is a pair of metrics on $T^{1,0}X$ solving
\begin{equation} \label{pair-eqn}
d \omega^2 = 0, \quad F_H \wedge \omega^2 = 0,
\end{equation}
where $\omega = i g_{j \bar{k}} dz^j \wedge d \bar{z}^k$. The balanced metrics $\what{g}_a$ and $g_t$ were constructed by Fu--Li--Yau \cite{FLY12}. The Hermitian-Yang--Mills metrics $\what{H}_a$ and $H_t$ were constructed by Collins, Yau and the second named author \cite{CPY21}. These non-K\"ahler equations are suggested by string theory \cite{strominger1986}, and proposed by S.-T. Yau and collaborators to geometrize conifold transitions \cite{liyau05,FLY12,CPY21,CPY23}. Near the ordinary double points, both metrics are close to the Candelas--de la Ossa \cite{CdlO90} K\"ahler Ricci-flat local models, but there are also global non-K\"ahler corrections. In other words, $g=H$ solves \eqref{pair-eqn} when they are both equal to a K\"ahler Ricci-flat metric, and although the global geometry is necessarily non-K\"ahler, the solution of \cite{FLY12,CPY21} approximately returns to the K\"ahler Ricci-flat solution on the local model.

\begin{rmk}
There is a third equation constraining $(g,H)$ which appears in heterotic string theory. This additional equation, named the anomaly cancellation equation, is conjectured to be solvable through conifold transitions \cite{liyau05,yaushape,FLY12} (see also e.g. \cite{dlO-Svanes,garciasurvey,garciamolina,TsengYau,picardsurvey} for a mathematical introduction to these equations). It is further conjectured that the pair $(g,H)$ can be rigidified in a suitable notion of cohomology class by a uniqueness property once this additional equation is imposed \cite{garcia2022canonical}. 
\end{rmk}

\begin{rmk}
Another approach to geometrizing conifold transitions via Chern-Ricci flat balanced metrics is proposed in \cite{tosatti2015non,toswein17,FWW} with recent progress by Giusti-Spotti \cite{giusti2023ak}. We also remark that the anomaly flow \cite{PPZCAG, PPZMathZ} is another mechanism for creating a canonical path of balanced metrics which has not yet been understood in the context of conifold transitions. 
\end{rmk}

Our main result is:

\begin{thm}
Let $\what{X}$ be a compact K\"ahler Calabi--Yau threefold with finite fundamental group. Let $\what{X} \rightarrow X_0 \rightsquigarrow X_t$ be a conifold transition. There exists a family of smooth metrics $(\what{X},\what{g}_a,\what{H}_a)$ for $0<a<1$ and $(X_t,g_t,H_t)$ for $0<|t|<\epsilon$ solving
\begin{equation*}
d \omega^2=0, \quad F_H \wedge \omega^2= 0
\end{equation*}
such that as the parameters $a$ and $t$ are varied, the geometries $(X,\what{g}_a,\what{H}_a)$ and $(X_t,g_t,H_t)$ vary continuously in the Gromov-Hausdorff sense and
\begin{align*}
& (\what{X}, \what{g}_a) \rightarrow (X_0,d_{g_0}) \leftarrow (X_t,g_t) \nonumber\\
& (\what{X}, \what{H}_a) \rightarrow (X_0,d_{H_0}) \leftarrow (X_t,H_t)
\end{align*}
as $a, t \rightarrow 0$ in the Gromov--Hausdorff topology. The length spaces $(X_0,d_{g_0})$ and $(X_0,d_{H_0})$ are induced by a limiting Hermitian--Yang--Mills structure on $((X_0)_{\mathrm{reg}},\omega_0,H_0)$.
\end{thm}

Our proof begins by analyzing the local models, which are K\"ahler Ricci-flat metrics on the small resolution and smoothing of the affine cone $\{ \sum z_i^2 = 0 \} \subset \mathbb{C}^4$. Once the local models are understood, we move on to the global balanced and Hermitian-Yang-Mills structures $(g,H)$. These global metrics add non-K\"ahler corrections to the K\"ahler Ricci-flat local models by solving a PDE on the global compact manifold: for the balanced metrics $\omega$, the PDE involves the 4th order Kodaira--Spencer operator, while for the metric $H$, the PDE is the Hermitian--Yang--Mills equation. Our analysis relies on suitable estimates for these equations along degenerations. To obtain continuity at $a=t=0$, the main step is to obtain diameter bounds tending to zero near the singular points, and the general approach is in the style of Song-Weinkove \cite{SW13,SW14}, where exceptional sets are contracted along a sequence of metrics.

\subsection*{Acknowledgements}
This research is supported in part by an NSERC Discovery Grant. The authors thank J. Bryan, J. Chen, and A. Fraser for helpful discussions. We also thank the referee for helpful comments and suggestions. S.P. thanks T.C. Collins and S.-T. Yau for previous collaborations on conifold transitions.

\section{Preliminaries}
\label{sect-prelims}

\subsection{The Gromov--Hausdorff Topology}
\label{subsect-GH-top'y}

The Gromov--Hausdorff topology was introduced in 1975 by Edwards \cite{Edw75}, and was then independently rediscovered by Gromov in the 1980's. Since then, it has been an indispensable tool in geometry. There has been growing interest in applications of the Gromov--Hausdorff topology to Calabi--Yau manifolds starting with the work of Gross--Wilson \cite{gross2000}, and we note in particular the use of this topology in studying the continuity of conifold transitions of Calabi--Yau threefolds (see \cite{RZ11, RZ11b, Song15}).

We will now introduce certain definitions and notions pertaining to Gromov--Hausdorff convergence of compact metric spaces. Other sources for this material include \eg \cite{BBI,Gro07,gross2000,Edw75,petersen2006}. We will implicitly assume that all our metric spaces in this section are compact, though generalizations exist for the non-compact case (\cf\,pointed Gromov--Hausdorff convergence).

Let $(X,d)$ be a compact metric space. For $A \subseteq X$ and $\eps > 0$, we set
\begin{equation*}
    B_\eps(A) = \bigcup_{x \in A} B_\eps(x),
\end{equation*}
where $B_\eps(x) = \{x' \in X \mid d(x',x) < \eps\}$ is the ball of radius $\eps$ around $x$.

\begin{defn}
\label{defn-eps-isom}
Let $(X,d_X)$ and $(Y,d_Y)$ be compact metric spaces and $\eps > 0$. A map $f \: X \rarr Y$ is called an $\eps$-isometry if
\begin{enumerate}[label = \roman*)]
    \item $|d_X(x,x') - d_Y(f(x),f(x'))| < \eps$ for all $x,x' \in X$, and

    \item $Y \subseteq B_\eps(f(X))$.
\end{enumerate}
\end{defn}

In general, $\eps$-isometries need not be injective or even continuous.

\begin{defn}
\label{defn-GH-dist}
The Gromov--Hausdorff distance $d_{\mathrm{GH}}$ between two compact metric spaces $(X,d_X)$ and $(Y,d_Y)$ is
\begin{equation*}
    d_{\mathrm{GH}}(X,Y) = \inf \{ \eps > 0 \mid \text{There exist } \eps\text{-isometries } f_1 \: X \rarr Y, \, f_2 \: Y \rarr X \}.
\end{equation*}
\end{defn}

The Gromov--Hausdorff distance $d_{\mathrm{GH}}$ defines a metric, and hence a topology, on the set $\cal{M}$ of isometry classes of compact metric spaces.

\begin{rmk}
\label{rmk-GH-symmetry}
We note that only one of the $\eps$-isometries in the previous definition is required as given an $\eps$-isometry $f_1 \: X \rarr Y$, one can construct a $3\eps$-isometry $f_2 \: Y \rarr X$. This in essence scales the Gromov--Hausdorff metric $d_{\mathrm{GH}}$ by a factor of $3$, but both generate the same topology on $\cal{M}$.
\end{rmk}

\subsection{Conifold Transitions}
\label{subsect-cfld-Trans}

Conifold transitions describe a process wherein one Calabi--Yau threefold is gradually deformed into another, passing through an intermediate space having cone singularities (\ie\,a conifold).

We briefly review certain facts about the geometry of conifold transitions. The exposition here will follow \cite{CPY21, FLY12, Fri91}. We begin with some definitions to fix the set-up of this document.

\begin{defn}
\label{defn-CY-3-fold}
A K\"ahler Calabi--Yau threefold $\what{X}$ is a compact complex manifold of complex dimension 3 with finite fundamental group, trivial canonical bundle, and admitting a K\"ahler metric.
\end{defn}

\begin{defn}
\label{defn-(-1,-1)-curve}
A $(-1,-1)$-curve $E \in \what{X}$ is a smooth rational curve $E \simeq \bb{P}^1$ such that the normal bundle $N_{E/\what{X}} \simeq \cal{O}(-1) \oplus \cal{O}(-1)$.
\end{defn}

Around a $(-1,-1)$-curve $E$, there exists an open neighbourhood $\what{U}$ which is biholomorphic to a neighbourhood of the zero section in the total space of the bundle
\begin{equation*}
    \what{V} = \cal{O}(-1) \oplus \cal{O}(-1) \rarr \bb{P}^1.
\end{equation*}
Given a collection of disjoint $(-1,-1)$ curves $\{ E_i \} \subset \what{X}$, we may contract them to points by a blowdown map $\pi \: \what{X} \rarr X_0$, where $X_0$ is a singular space with isolated singular points at $s_i = \pi(E_i)$.

In more detail, we identify a neighborhood of $E_i$ with the model space $\what{V}$ and the blowdown map $\pi$ sends the complement of the zero section biholomorphically to the complement of the origin in the conifold
\begin{equation*}
    V_0 = \Big\{z \in \bb{C}^4 \mid \sum_i z_i^2 = 0 \Big\}.
\end{equation*}
This map $\pi$ can be holomorphically extended to all of $\what{V}$ by sending the zero section to the origin in $V_0$. The map $\pi$ near $E_i$ can be made explicit and we give the expression later in \eqref{pi-defn}. The result is that $X_0$ has isolated ordinary double point (ODP) singularities $s_i$ with local neighborhoods biholomorphic to $0 \in V_0$.

We next discuss how to smooth the singular space $X_0$ by deforming its holomorphic structure. The local model is a singular variety $V_0$ which can be smoothed by considering the space
\begin{equation*}
    \cal{V} = \Big\{ (z,t) \in \bb{C}^4 \times \bb{C} \mid \sum_i z_i^2 = t \Big\}.
\end{equation*}
The fiber over $t$ is denoted $V_t$ (considered as a subset of $\bb{C}^4$) and is smooth for all $t \neq 0$.
\begin{equation*}
    V_t = \Big\{z \in \bb{C}^4 \mid \sum_i z_i^2 = t \Big\}.
\end{equation*}
This is the local model which we would like to achieve globally on $X_0$. A result of R. Friedman \cite{Fri86} gives a condition describing the existence of a smoothing.

\begin{figure}
    \centering
    \def\svgwidth{1.1\columnwidth}
    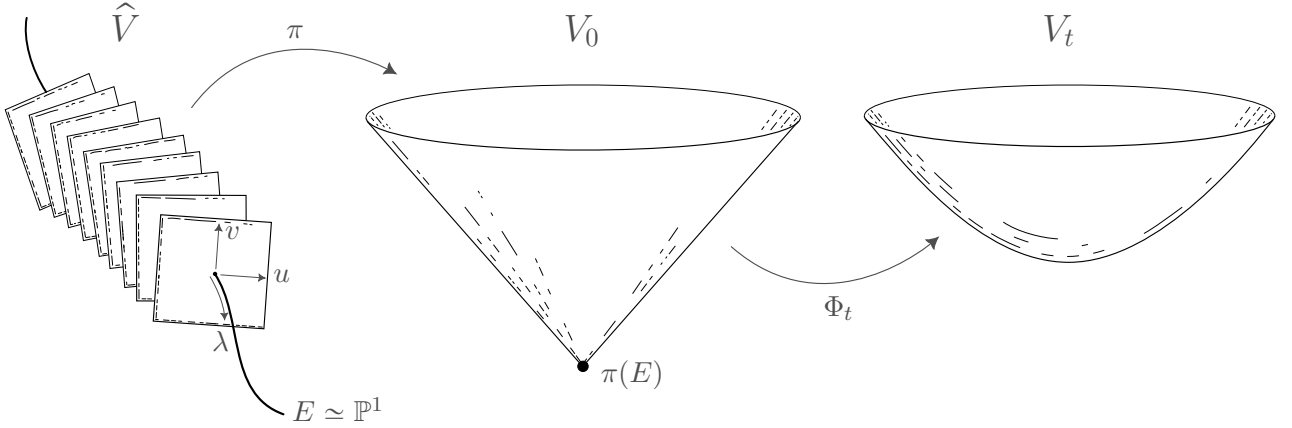
    \caption{Local model of a conifold transition.}
    \label{fig-cfld-trans}
\end{figure}

\begin{thm}[R. Friedman \cite{Fri86,Fri91}]
\label{thm-smthing}
Suppose $\what{X}$ is a K\"ahler Calabi--Yau threefold and let $E_1,\ldots,E_k$ be disjoint $(-1,-1)$-curves. Let $\pi$ be the blowdown map that contracts each $E_i$, resulting in the singular space $X_0$ with ODP singularities $s_i = \pi(E_i)$. There exists a first order deformation of $X_0$ smoothing each $s_i$ if and only if there exists a relation
\begin{equation*}
    \sum_i \lambda_i [E_i] = 0 \text{ in } H^2(\what{X},\bb{R})
\end{equation*}
with each $\lambda_i \neq 0$.
\end{thm}

It has been shown that the first order deformations from the above theorem integrate to genuine smoothings; see \cite{Kaw92, Ran92, Tia92}. Thus assuming the condition of Theorem \ref{thm-smthing} holds, we get a holomorphic familily
\begin{equation*}
    \mu \: \cal{X} \rarr \Delta 
\end{equation*}
where $\Delta \subset \mathbb{C}$ denotes the unit complex disk such that the fibers $X_t = \mu^{-1}(t)$ are smooth complex manifolds for $t \neq 0$ and $X_0 = \mu^{-1}(0)$. A result of Kas--Schlessinger \cite{KS72} shows that the family $\cal{X}$ is locally biholomorphic to the model $\cal{V}$ near each ODP. It can be shown that the complex manifolds $X_t$ have trivial canonical bundle; see \cite{Fri86} for an algebraic proof or \cite{CPY23} for a differential geometric proof.

\begin{defn}
\label{defn-cfld-trans}
Let $\what{X}$ be a K\"ahler Calabi--Yau threefold. A conifold transition starting from $\what{X}$, denoted $\what{X} \rarr X_0 \rightsquigarrow X_t$, consists of a holomorphic map $\pi \: \what{X} \rarr X_0$ and a family $\mu \: \cal{X} \rarr \Delta$ with $X_0 = \mu^{-1}(0)$ such that
\begin{enumerate}[label = \roman*)]
    \item the map $\pi \: \what{X} \rarr X_0$ contracts a collection of disjoint $(-1,-1)$-curves $E_1,\ldots,E_k \subseteq \what{X}$ to isolated ODP singularities $s_1,\ldots,s_k \in X_0$, and $\pi$ is a biholomorphism between $\what{X} \bs (E_1 \cup \ldots E_k)$ and $X_0 \bs \{s_1,\ldots,s_k\}$; and

    \item the total space $\cal{X}$ is a smooth complex fourfold with a proper flat map $\mu \: \cal{X} \rarr \Delta$, where $X_0 = \mu^{-1}(0)$ and $X_t = \mu^{-1}(t)$ are smooth complex manifolds for $t \neq 0$.
\end{enumerate}
\end{defn}

It is known that the K\"{a}hler condition is not necessarily preserved along a conifold transition. For a concrete example, suppose $\what{X}$ is a quintic threefold and choose a pair of disjoint $(-1,-1)$-curves $E_1$, $E_2$ (for the existence of such curves, see e.g. \cite{clemens1983}). Since $b_2(\what{X})=1$, these satisfy Friedman's condition. Thus a conifold transition exists, and since the generator of second homology has been sent to a point, we have $b_2(X_t)=0$. We see that $X_t$ may not support any K\"ahler metric, even if the initial $\what{X}$ is a projective threefold. For further examples of K\"ahler to non-K\"ahler conifold transitions we refer to \cite{Fri91,lutian}, and for the study of Hodge structures through such a process, see \cite{Fri19}.

To geometrize the parameter space of Calabi--Yau threefolds connected by conifold transitions, we must therefore look for special non-K\"ahler metrics. This program was initiated by Fu--Li--Yau \cite{FLY12}. The inspiration comes from supersymmetric constraints in string theory. K\"ahler Calabi--Yau metrics satisfy the system of supersymmetric constraints when the $H$-flux is taken to be zero \cite{CHSW}. As pointed out in \eg \cite{reid87}, Chapter 4 of \cite{hubschbesti}, or Section 6 of \cite{COGP}, a degeneration and resolution may connect a K\"ahler threefold to a non-K\"ahler space, and so it is necessary to look for more general solutions to the supersymmetric constraints with non-zero $H$-flux. These constraints were worked out by Strominger \cite{strominger1986} and imply the following two equations:

\begin{itemize}
    \item $X$ admits a balanced metric $\omega$. A Hermitian metric $g$ on $T^{1,0}X$ over a complex manifold $X$ of dimension $n$ is balanced if
    \begin{equation} 
    \label{defn-balanced}
        d \omega^{n-1} = 0.
    \end{equation}
    Here $\omega$ is the $(1,1)$-form associated to $g$ via $\omega= \sqrt{-1} g_{j \bar{k}} dz^j \wedge d \bar{z}^k$. Various properties of balanced metrics were explored by Michelsohn \cite{Michelsohn}.

    \item $X$ admits a Hermitian--Yang--Mills metric. A Hermitian metric $H$ on $T^{1,0}X$ is Hermitian--Yang--Mills with respect to a balanced metric $\omega$ if
    \begin{equation} 
    \label{defn-HYM}
        F \wedge \omega^{n-1} = 0.
    \end{equation}
    The Chern curvature of a metric $H$ is denoted by $F \in \Lambda^{1,1}({\rm End} \, T^{1,0}X)$, and is given by $F=\bar{\partial}(\partial H H^{-1})$. The criterion for the solvability of this equation over a general holomorphic bundle is given by the Donaldson--Uhlenbeck--Yau theorem \cite{donaldson1985,uhlenbeckyau86} in the K\"ahler case, with an extension by Li--Yau \cite{liyau87} for non-K\"ahler metrics.
\end{itemize}

When $X$ is a K\"ahler threefold, Yau's theorem \cite{yau78} gives the existence of a K\"ahler Ricci-flat metric $g_{\rm CY}$. The above equations are then solved with $g=H=g_{\rm CY}$. More generally, let $X_t$ be a complex manifold connected to a K\"ahler threefold via a conifold transition $\what{X} \rightarrow X_0 \rightsquigarrow X_t$. The main results of \cite{FLY12} and \cite{CPY21} give the existence of a pair $(g,H)$ solving the supersymmetric equations \eqref{defn-balanced} and \eqref{defn-HYM}. 

\begin{rmk}
\label{rmk-aux-gauge-bdle}
    The Hermitian--Yang--Mills equation may in principle also be solved over an auxiliary gauge bundle $E$, but the mechanism under which a conifold transition creates a stable holomorphic vector bundle $E_t \rightarrow X_t$ is not understood except for the case at hand, which is when $E_t = T^{1,0} X_t$. There is a proposal by Anderson--Brodie--Gray \cite{anderson2022} in the string theory literature on how such general bundles may appear on the other side. There is also work of Chuan \cite{Chuan} on the Hermitian--Yang--Mills equation on a gauge bundle $E$ with the additional assumption that $E$ is locally a trivial bundle through the singularities of a conifold transition.
\end{rmk}

In the remainder of this preliminary section, we recall various metrics which can be defined both globally and on the local models, and state the main results of this paper, which state that conifold transitions, when bestowed with these metrics, describe a continuous path in $\mathcal{M}$ with respect to the Gromov--Hausdorff topology.

\subsection{Metrics on Small Resolutions}
\label{subsect-Metrics-SR}

\subsubsection{Candelas--de la Ossa Metrics on the Local Model}
\label{sect-CdlO-SR}

Consider the space $\what{V}$, which is the total space of the bundle
\begin{equation*}
    \what{V} = \cal{O}(-1) \oplus \cal{O}(-1) \rarr \bb{P}^1.
\end{equation*}
On this space, we have two trivializations
\begin{equation*}
    (U, (\lambda, u,v)) \text{ and } (U', (\lambda', u', v')),
\end{equation*}
with transition functions given by
\begin{equation*}
    \lambda' = \lambda^{-1}, \, u' = \lambda u, \, v' = \lambda v.
\end{equation*}
Note that $\lambda$ is the coordinate on the base space $\mathbb{P}^1$, while $u,v$ are fibre coordinates.

It will be convenient to define the well-defined radius function $r: \what{V} \rightarrow [0,\infty)$ given by
\begin{equation*}
    r(\lambda, u, v) = (1 + |\lambda|^2)^{\frac{1}{3}} (|u|^2 + |v|^2)^{\frac{1}{3}}.
\end{equation*}
Without the power of $\frac{1}{3}$, this function measures the distance squared from a point to the zero section $E \simeq \bb{P}^1$ along the fibers using the Fubini--Study metric $\what{\omega}_{FS}$. The power is introduced so that this radius function coincides with the radius of the Calabi--Yau cone metric on the blowdown, and this will be discussed later.

The space $(\what{V}, r)$ is also equipped with a family of scaling maps. Namely, for $R>0$, we have the map $S_R: \what{V} \rightarrow \what{V}$ given by
\begin{equation*}
    S_R(\lambda,u,v) = (\lambda,R^{\frac{3}{2}}u, R^{\frac{3}{2}} v).
\end{equation*}
The radius behaves as it should under the scaling, as we have
\begin{equation*}
    r \circ S_R = R \cdot r.
\end{equation*}
In \cite{CdlO90}, Candelas--de la Ossa look for a K\"ahler Ricci-flat metric $\what{\omega}_{\co,1}$ on $\what{V}$ of the form
\begin{equation*}
    \what{\omega}_{\co,1} = \sqrt{-1} \partial \overline{\partial} f(r^3) + 4 \what{\omega}_{\mathrm{FS}},
\end{equation*}
where $\what{\omega}_{\mathrm{FS}}$ is the Fubini--Study metric on $\mathbb{P}^1$, and $f(x)=f(r^3)$ is a smooth function. They show that imposing the condition of K\"ahler--Ricci flatness yields the following first order ODE for $f$:
\begin{equation*}
x(f'(x))^3+6(f'(x))^2=1.
\end{equation*}
The solution admits an expansion \cite{CPY21} for $x \gg 1$ given in terms of $r=x^{\frac{1}{3}}$ by
\begin{equation*}
f =c_0 r^2 + c_1 \log r + c_2 r^{-2} + c_3 r^{-4} + \cdots
\end{equation*}
for constants $c_0, c_1, c_2, \cdots$. Thus, after rescaling $\what{\omega}_{\co,1}$ such that $c_0 = \frac{1}{2}$, we have the following expansion for large radius $r \gg 1$:
\begin{equation} 
\label{co1-expansion}
    \what{\omega}_{\co,1} - \frac{1}{2} \sqrt{-1} \partial \bar{\partial} r^2 = c_{-1} \what{\omega}_{\mathrm{FS}} + c_1 \sqrt{-1} \partial \bar{\partial} \log r + c_2 \sqrt{-1} \partial \overline{\partial} r^{-2} + c_3 \sqrt{-1} \partial \overline{\partial} r^{-4} + \ldots.
\end{equation}
Next, we note that the space $\what{V}$ can be regarded as a small blow-up of the space
\begin{equation*}
    V_0 = \Big\{ z \in \bb{C}^4 \mid \sum_i z_i^2 = 0 \Big\}.
\end{equation*}
There is a (scaled) blow-down map $\pi: \what{V} \rightarrow V_0$ such that $\pi^{-1}(0)$ is the zero section $E=\{r=0\}$, and the restriction
\begin{equation*}
    \pi: \hat{V} \, \backslash \, E \rightarrow V_0 \, \backslash  \,\{ 0 \}
\end{equation*}
is biholomorphic. The map $\pi$ has the explicit expression
\begin{equation}
\label{pi-defn}
    \pi (\lambda, u, v) = \Bigg( \frac{\lambda v + u}{\sqrt{2}}, -\sqrt{-1} \cdot \frac{\lambda v - u}{\sqrt{2}}, -\sqrt{-1} \cdot \frac{v + \lambda u}{\sqrt{2}}, -\frac{v - \lambda u}{\sqrt{2}} \Bigg).
\end{equation}
Likewise, away from the singularity at the origin, we can see that
\begin{equation*}
    \pi^{-1} (z_1, z_2, z_3, z_4) = \Bigg( \frac{z_1 - \sqrt{-1} z_2}{\sqrt{2}}, \frac{\sqrt{-1} (z_3 + \sqrt{-1} z_4)}{\sqrt{2}}, \sqrt{-1} \cdot \frac{z_3 - \sqrt{-1} z_4}{z_1 - \sqrt{-1} z_2} \Bigg).
\end{equation*}


The function $r$ on $\what{V}$ becomes $\| z \|^{\frac{2}{3}}$ on $V_0$ via the identification $\pi$, in the sense that $r(\lambda, u, v) = \| \pi (\lambda, u, v)\|^{\frac{2}{3}}$.
For this reason, we will also denote
\begin{equation*}
r = \| z \|^{\frac{2}{3}}    , \quad r: V_0 \rightarrow [0,\infty).
\end{equation*}
The space $V_0$ admits a Calabi--Yau cone metric
\begin{equation*}
    \omega_{\co,0} = \frac{1}{2} \sqrt{-1} \partial \bar{\partial} r^2.
\end{equation*}
We briefly discuss the cone metric geometry on $(V_0, \omega_{\co,0})$. Observe that $V_0$ is closed under scalar multiplication and addition, so that $V_0$ is a cone. The metric $\omega_{\co,0}$ is well-known \cite{CdlO90} to be K\"ahler Ricci-flat and is a cone metric over the link
\begin{equation*}
    L = \{z \in V_0 \mid r(z) = 1\},
\end{equation*}
and we may write
\begin{equation}
\label{eqn-cone-metric}
    g_{\co,0} = dr \otimes dr + r^2 \cdot g_L,
\end{equation}
where $g_L$ is the pullback of a metric on $L$.

The link $L$ is $S^3 \times S^2$. To see this, we express the defining condition $\sum_i z_i^2 = 0$ of $V_0$ in real coordinates $x_i, y_i$ such that $z_i = x_i + \sqrt{-1}y_i$ for each $i \in \{1,2,3,4\}$. We obtain
\[
0 = \norm{x}^2 - \norm{y}^2 +2\sqrt{-1}\langle x, y \rangle,
\]
where $x=(x_1,x_2,x_3,x_4) \in \mathbb{R}^4$ and $y=(y_1,y_2,y_3,y_4) \in \mathbb{R}^4$. Expressed in these terms, $V_0$ is the set of all $(x,y) \in \mathbb{R}^8 \cong \mathbb{R}^4 \oplus \mathbb{R}^4$ on which $\norm{x} = \norm{y}$ and $\langle x, y \rangle = 0$. Fixing $r^3=\norm{x}^2 + \norm{y}^2=2$ implies that $\norm{x}=\norm{y}=1$. In particular, we have $x \in S^3 \subset \mathbb{R}^4$. Then for each such choice of $x$, the conditions $\langle x, y \rangle=0$ and $\norm{y} = 1$ imply that $y$ is in the unit $2$-sphere centered at $0$ in the tangent space $T_x S^3$. Thus, the set $\{z \in V_0 \mid r(z)=2^{\frac{1}{3}}\}$ is diffeomorphic to the unit sphere bundle contained in the tangent bundle $TS^3$, which is trivial. Thus $\{z \in V_0 \mid r(z)=2^{\frac{1}{3}}\} \cong S^3 \times S^2$, and by rescaling we have $L \cong S^3 \times S^2$ as well.

Returning to $(\what{V}, \what{g}_{\co,1})$, we can rescale the area of the zero section $E \simeq \mathbb{P}^1$ to obtain a 1-parameter family of metrics. We will refer to this family of metrics as the Candelas--de la Ossa metrics $\what{g}_{\co,a}$ on the small resolution. The metrics $\what{g}_{\co,a}$ satisfy the following important properties:
\begin{enumerate}[label = (CO SR \Roman*), align = left]
    \item \label{prop-SR-I} \textit{Normalization:} For $a > 0$, we have
    \begin{equation*}
        \what{g}_{\co,a} = a^2 \cdot S_{a^{-1}}^* (\what{g}_{\co,1}).
    \end{equation*}
    \item \label{prop-SR-II} \textit{Asymptotically Conical Decay:} There exists $C > 0$ independent of $a$ such that for all $a > 0$,
    \begin{equation*}
        |(\pi^{-1})^* (\what{g}_{\co,a}) - g_{\co,0}|_{g_{\co,0}} \leq C a^2 r^{-2}.
    \end{equation*}
    The asymptotic decay can be derived from \eqref{co1-expansion} for $a=1$. Pulling-back the estimate when $a=1$ by $S^*_a$ gives the estimate for general $a$. The estimate implies that the Candelas--de la Ossa metrics $\what{g}_{\co,a}$ converge uniformly to the cone metric $g_{\co,0}$ on compact sets away from the zero section $E$. 
\end{enumerate}

For $R>0$, we will denote by $\what{T}(R)$ the ``tube''
\[
\what{T}(R) := \{r \leq R\} \subseteq \what{V}
\]
around the zero section $E \simeq \bb{P}^1$.
On the cone $V_0$, we will analogously refer to the ``disc"
\[
D_0(R) := \{r \leq R\} \subseteq V_0.
\]
of radius $R$ around $0$.

Let $\what{d}_{\co,a}$ and $d_{\co,0}$ denote the induced distance functions on $\what{X}$ and $X_0$ by $\what{g}_{\co,a}$ and $g_{\co,0}$ respectively. As warm-up for our proof of the convergence of metrics on the compact threefold, we will present a proof of convergence of the local models.

\begin{thm}
    \label{thm-CdlO-SR-GH}
     The spaces $( \what{T}(1),\what{d}_{\co,a} )$ converge in the Gromov--Hausdorff topology as $a \to 0$ to the space $(D_0(1), d_{\co,0})$.
\end{thm}

A proof of this fact also follows from the PDE estimates in \cite{CGT22} for general asymptotically conical Calabi--Yau metrics on small resolutions.

\subsubsection{Balanced Metrics on the Small Resolution}
\label{subsubsect-balanced-SR}

Next, we state the properties that we will need from the Fu--Li--Yau metrics on the compact Calabi--Yau threefold $\what{X}$. Let $\what{\omega}_{\rm CY}$ be a K\"ahler metric on the K\"ahler threefold $\what{X}$. The Fu--Li--Yau \cite{FLY12} gluing construction (see also \cite{CPY21}, \cite{Chuan} for further details on $\what{\omega}_{\FLY,a}$ for $a \neq 0$) produces a sequence of metrics $\what{\omega}_{\FLY,a}$ for $0 \leq a \leq 1$ such that
\begin{equation*}
    d \what{\omega}_{\FLY,a}^2 = 0, \quad [\what{\omega}_{\FLY,a}^2]=[\what{\omega}_{\rm CY}^2] \in H^4(\what{X},\mathbb{R}).
\end{equation*}
For the purposes of this paper, we will mainly make use of the two following properties: 
\begin{enumerate}[label = (FLY SR \Roman*),align = left]
    \item \label{prop-FLY-I} \textit{Local Model:} 
    there exists $\delta>0$ and $R>1$ such that for all $0 \leq a \leq 1$, we have
\begin{equation*}
 \what{\omega}_{\FLY,a}|_{ \{ r< \delta \}} = R \cdot \what{\omega}_{\co,a}.
\end{equation*}
Here the function $r: \what{X} \rightarrow [0,\infty)$ extends the local functions $r$ defined on a neighborhood of the curves $E_i \subset \what{V}$ to the whole compact manifold $\what{X}$ such that the set $\{ r < \delta \}$ consists of small disjoint open neighborhoods containing the $(-1,-1)$-curves $E_1, \dots, E_k$.
    \item \label{prop-FLY-II} \textit{Uniform Convergence:} For any compact set $K \subset \what{X} \backslash ( E_1 \cup \dots \cup E_k)$, the sequence $\what{\omega}_{\FLY,a}$ converges uniformly to $\what{\omega}_{\FLY,0}$ as $a \rightarrow 0$ on $K$.
\end{enumerate}

For each $E_i \simeq \mathbb{P}^1$, these metrics satisfy
\begin{equation}
    \int_{\mathbb{P}^1} \what{\omega}_{\FLY,a} \rightarrow 0, \quad a \rightarrow 0.
\end{equation}
The limiting metric $\what{\omega}_{\FLY,0}$ is singular on the curves $E_1, \dots E_k$, and only defines a genuine metric on $\what{X} \backslash (E_1 \cup \dots \cup E_k)$.

Let $\pi: \what{X} \rightarrow X_0$ be the contraction of the curves, and let $s_i = \pi(E_i)$ be the singular points of $X_0$. We will write $(X_0)_{\rm reg} = X_0 \backslash \{ s_1, \dots, s_k \}$. Since $\what{X} \backslash (E_1 \cup \dots \cup E_k) \simeq (X_0)_{\rm reg}$, the limiting metric $\what{\omega}_{\FLY,0}$ defines a Riemannian structure $( (X_0)_{\mathrm{reg}}, \, \omega_{\FLY,0} )$ with conical singularities. This induces a distance function $d_{\FLY,0}$ on $X_0$.

This brings us to one of our main theorems:
\begin{thm}
    \label{thm-FLY-SR-GH}
    The compact metric spaces $(\what{X},\what{d}_{\FLY,a})$ converge to $(X_0,d_{\FLY,0})$ in the Gromov--Hausdorff topology as $a \to 0$.
\end{thm}

\subsubsection{Hermitian--Yang--Mills Metrics on the Small Resolution}
\label{subsubsect-HYM-SR}

We review the relevant properties of the sequence of Hermitian--Yang--Mills metrics from \cite{CPY21}. Recall that our initial threefold $\what{X}$ is K\"{a}hler Calabi--Yau and simply connected. By dimensional considerations, an application of the de Rham Decomposition Theorem \cite{yau93} implies that $(\what{X}, \what{\omega}_{\rm CY})$ satisfies the stability condition
\begin{equation*}
    \frac{1}{\rank F} \int_{\what{X}} c_1(F) \wedge \what{\omega}^2_{\rm CY} < 0
\end{equation*}
for all torsion-free coherent proper subsheaves $F \subseteq T^{1,0}\what{X}$. The Fu--Li--Yau metric $\what{\omega}_{\FLY,a}$ and the Calabi--Yau metric $\what{\omega}_{\rm CY}$ have the same squared cohomology class and so
\begin{equation*}
    \frac{1}{\rank F} \int_{\what{X}} c_1(F) \wedge \what{\omega}^2_{\FLY,a} < 0.
\end{equation*}

It follows that $T^{1,0}\what{X}$ is stable with respect to each of the Fu--Li--Yau metrics. The Li--Yau \cite{liyau87} generalization of the Donaldson--Uhlenbeck--Yau Theorem \cite{donaldson1985, uhlenbeckyau86} to Gauduchon metrics yields a family of metrics $\what{H}_a$ satisfying 
\begin{equation} \label{hym-a-norma}
    F_{\what{H}_a} \wedge \what{\omega}^2_{\FLY,a} = 0, \quad  \int_{\hat{X}} \log \frac{\det \what{H}_a}{\det \what{g}_{\FLY,a}} \, d {\rm vol}_{g_{\FLY,a}} = 0.
\end{equation}
This sequence $\what{H}_a$ satisfies the following estimates:
\begin{propn}[Collins--Picard--Yau \cite{CPY21}]
\label{propn-HYM-ests-SR}
There exists constants $C, C_k > 0$ such that the Hermitian--Yang--Mills metrics $\what{H}_a$ satisfy
\begin{equation*}
    C^{-1} \cdot \what{g}_{\FLY,a} \leq \what{H}_a \leq C \cdot \what{g}_{\FLY,a},
\end{equation*}
\begin{equation*}
    |\nabla^k \what{H}_a|_{\what{g}_{\FLY,a}} \leq C_k r^{-k}.
\end{equation*}
\end{propn}

The metric $H_0$ on $(X_0)_{\mathrm{reg}}$ can be constructed as the limit of these metrics $\what{H}_a$. This was done in \cite{CPY21} by taking a subsequence of $\{ \what{H}_a \}$. In Appendix \ref{appendix-hym}, we will show that these estimates imply that the full sequence converges on compact sets. Therefore there exists a Hermitian--Yang--Mills metric $H_0$ over the singular space $X_0$ such that 
\begin{equation*}
    F_{H_0} \wedge \omega^2_{\mathrm{FLY,0}} = 0, \quad C^{-1} \cdot g_{\FLY,0} \leq H_0 \leq C \cdot g_{\FLY,0}
\end{equation*}
and for any compact set $K \subset \what{X} \backslash (E_1 \cup \dots \cup E_k)$, the sequence $\what{H}_a$ converges uniformly to $H_0$ on $K$ as $a \rightarrow 0$. Going beyond compact subsets, we will prove:

\begin{thm}
    \label{thm-HYM-SR-GH}
    The compact metric spaces $(\what{X},\what{d}_{\what{H}_a})$ converge to $(X_0,d_{H_0})$ in the Gromov--Hausdorff topology as $a \to 0$.
\end{thm}

\subsection{Metrics on Smoothings}
\label{subsect-metrics-SM}

\subsubsection{Candelas--de la Ossa Metrics on the Local Model}
\label{sect-CdlO-SM}

Next, for $t \neq 0$, we consider the smooth submanifolds $V_t \subset \mathbb{C}^4$ defined by
\begin{equation*}
    V_t = \Big\{ z \in \bb{C}^4 \mid \sum_i z_i^2 = t \Big\}.
\end{equation*}
We have the usual norm $\|z\|^2$ on $V_t$ induced from $\bb{C}^4$ given by
\begin{equation*}
    \|z\|^2 = \sum_{i=1}^4 |z_i|^2.
\end{equation*}
Candelas--de la Ossa also constructed metrics $\omega_{\co,t}$ on the smoothings $V_t$ \cite{CdlO90}. The metrics are obtained by looking for potentials of the form 
\begin{equation*}
  \omega_{\co,t} = \sqrt{-1} \del \delbar \varphi_t, \quad   \varphi_t = f_t(\| z \|^2),
\end{equation*}
and imposing the Ricci-flat condition which reduces to solving a differential equation for $f_t$. These metrics $\omega_{\co,t}$ are asymptotic to the cone geometry $(V_0,\omega_{\co,0})$, and we will make this more precise below.

As we did in the previous section, we define a radius function $r: V_t \rightarrow (0,\infty)$ by
\begin{equation*}
   r(z) = \| z \|^{\frac{2}{3}}.
\end{equation*}
Note that the condition $\sum_{i=1}^4 z_i^2 = t$ implies that $r(z) \geq |t|^{\frac{1}{3}}$ for all $z \in V_t$.
 
When $R > 0$, we may also define scaling maps $S_R \: \bb{C}^4 \rarr \bb{C}^4$ by
\begin{equation*}
    S_R(z) = R^{\frac{3}{2}} \cdot z.
\end{equation*}
The scaling map $S_R$ sends $V_\rho$ to $V_{R^3 \cdot \rho}$ and satisfies
\begin{equation*}
    r \circ S_R = R \cdot r \quad \text{ and } \quad S_R^* (\omega_{\co,0}) = R^2 \cdot \omega_{\co,0}.
\end{equation*}

Unlike the case of the small resolutions, the metrics $\omega_{\co,t}$ and $\omega_{\co,0}$ all lie on different spaces. In order to compare them (and obtain an analog of Property \ref{prop-SR-II}), we use the map $\Phi \: \bb{C}^4 \bs \{0\} \rarr \bb{C}^4$ given by
\begin{equation*}
    \Phi(z) = z + \frac{\overline{z}}{2 \|z\|^2}.
\end{equation*}
Routine computations show that $\Phi$ maps $V_0 \bs \{0\}$ into $V_1$. This map is not injective in general, but is a diffeomorphism when restricted to the set $\{z \in V_0 \mid r(z) > (\frac{1}{2})^{\frac{1}{3}} \}$ with image $\{z \in V_1 \mid r(z) > 1 \}$. In the sequel, $\Phi$ will refer to this restricted map and we will use results involving the map $\Phi$ as required. More details regarding them can be found in Appendix \ref{app-Phi}.

By composing $\Phi$ with $S_R$ for appropriate choices of $R$, we get a map from $V_0 \bs \{0\}$ to $V_t$. More precisely, let
\begin{equation}
\label{eqn-Phi_t}
    \Phi_t = S_{t^{\frac{1}{3}}} \circ \Phi \circ S_{t^{-\frac{1}{3}}}.
\end{equation}
Here we used the scaling maps to scale by a root of a non-zero complex number. Note that although $S_{t^{\frac{1}{3}}}$ depends on a choice of branch, the map $\Phi_t$ does not. We can compute that for $z \in V_0$, then
\begin{equation*}
    r(\Phi_t(z)) = \Big( (r(z))^3 + \frac{|t|^2}{4 (r(z))^3} \Big)^{\frac{1}{3}}.
\end{equation*}
In view of this identity, it will be convenient later on (especially in  Section \ref{subsect-ests-SM}) to use the notation
\begin{equation*}
    \beta_{t,\rho} = \Big( \rho^3 + \frac{|t|^2}{4\rho^3} \Big)^{\frac{1}{3}}.
\end{equation*}
We also can check that $\Phi_t$ is a diffeomorphism from $\{z \in V_0 \mid r(z) > (\frac{|t|}{2})^{\frac{1}{3}} \}$ to $\{z \in V_t \mid r(z) > |t|^{\frac{1}{3}} \}$. We use $\Phi_t$ to refer to the restricted map.

Using the maps $S_R$ and $\Phi_t$, we have the analogous properties of the Candelas--de la Ossa metrics $g_{\co,t}$ on the smoothings:

\begin{enumerate}[label = (CO SM \Roman*), align = left]
    \item \label{prop-SM-I} \textit{Normalization:} For $t \neq 0$, we have
    \begin{equation*}
        g_{\co,t} = |t|^{\frac{2}{3}} \cdot S_{t^{-\frac{1}{3}}}^* (g_{\co,1}).
    \end{equation*}
    The metrics $g_{\co,t}$ do not depend on the branch choice of $t^{-\frac{1}{3}}$.
    \item \label{prop-SM-II} \textit{Asymptotically Conical Decay:} There exists $C > 0$ independent of $t$ such that for all $t \neq 0$,
    \begin{equation*}
        |(\Phi_t)^* (g_{\co,t}) - g_{\co,0}|_{g_{\co,0}} \leq C |t| r^{-3}.
    \end{equation*}
    A consequence of this is that the metrics $g_{\co,t}$ approach $g_{\co,0}$ on compact sets away from the cone singularity as $t \rarr 0$.
\end{enumerate}

The proof of the asymptotically conical decay estimate can be found in \eg \cite{conlonhein}, where the estimate is given on $(V_1,g_{\co,1})$:
\begin{equation}
    | \Phi^* g_{\co,1} - g_{\co,0}|_{g_{\co,0}} \leq C r^{-3},
\end{equation}
and the estimate for $(V_t, g_{\co,t})$ follows by pulling-back via $S_{t^{-\frac{1}{3}}}$.

Let
\begin{equation*}
    D_t(R) := \{r \leq R \} \subseteq V_t.
\end{equation*}
be ``discs" of radius $R$ in $V_t$.

As warm-up to our result on continuity of the global non-K\"ahler geometry, we will present the following well-known convergence of the local models:

\begin{thm}
\label{thm-CdlO-SM-GH}
    The spaces $\Big(D_t(\beta_{t,1}), d_{\co,t} \Big)$ converge in the Gromov--Hausdorff topology as $t \to 0$ to the space $\Big( D_0(1), d_{\co,0} \Big)$.
\end{thm}

\subsubsection{Balanced Metrics on the Smoothings}
\label{subsubsect-balanced-SM}

We return to the global setting, where we have a holomorphic family $\mu: \mathcal{X} \rightarrow \Delta$ with smooth fibers $X_t = \mu^{-1}(t)$ for $t \neq 0$ and central fiber $X_0$ with singularities $\{s_1, \dots , s_k \}$ which are locally of the form $0 \in V_0$. If $X_0$ comes from holomorphic contraction $\pi: \what{X} \rightarrow X_0$ of $(-1,-1)$-curves, the compact complex manifolds $X_t$ may not support any K\"ahler metric. Fu--Li--Yau \cite{FLY12} prove that $X_t$ admits balanced metrics.

In order to define metrics on the smoothings $X_t$, we first need to extend the local maps $r$ and $\Phi_t$ from the previous section to global objects. 

For this, we note that there are disjoint open sets $\mathcal{U}_i \subset \mathcal{X}$ containing each singular point $s_i$ such that $\mathcal{U}_i$ is identified with 
\begin{equation*}
0 \in \mathcal{U} \subset \Big\{ (z,t) \in \bb{C}^4 \times \bb{C} \mid \sum_i z_i^2 = t \Big\}.
\end{equation*}
We can then extend the local functions $r = \| z \|^{\frac{2}{3}}$ on $\mathbb{C}^4$ to a global function $r: \mathcal{X} \rightarrow [0,\infty)$ with $r^{-1}(0) = \{ s_1, \dots, s_k \}$.

Next, we can extend the local maps $\Phi_t$ to global diffeomorphisms. 
\begin{equation*}
    \Phi_t: X_0 \cap \Big\{ r(z) > \Big(\frac{|t|}{2} \Big)^{\frac{1}{3}} \Big\} \rightarrow X_t \cap \{ r(z) > |t|^{\frac{1}{3}} \}
\end{equation*}
such that $\Phi_t$ is the model smoothing \eqref{eqn-Phi_t} on the local sets $\mathcal{U}_i$. This can be done by taking a horizontal lift $\xi$ of the vector field $\frac{\partial}{\partial t}$ on $\Delta$ which agrees with the vector field generating the model smoothing. Then flowing by the lifted vector field $\xi$ on $\mathcal{X}$ gives $\Phi_t$.

The Fu--Li--Yau construction \cite{FLY12} leads to a sequence $\omega_{\FLY,t}$ of hermitian metrics on $X_t$ solving
\begin{equation*}
    d \omega_{\FLY,t}^2 = 0.
\end{equation*}
These are obtained by: 1) a pullback and gluing construction, followed by 2) a perturbation step to ensure the balanced condition. The metric in step 1 is denoted $g_t$ and the metric in step 2 is denoted $g_{\FLY,t}$.

\par $\bullet$ Step 1. The expression for $\omega_t$ from \cite{FLY12} is
    \begin{equation} \label{fly-pre-per}
    \begin{aligned}
        \omega_t^2 &= pr_t^{2,2} \Bigg[ (\Phi_t^{-1})^* \Big(\omega_{\FLY,0}^2 - \sqrt{-1} \del \delbar (\rho_0 \cdot f_0 (\|z\|^2) \cdot \sqrt{-1} \del \delbar f_0 (\|z\|^2)) \Big) \\
        &\qquad \qquad + \sqrt{-1} \del \delbar \Big( \rho_t \cdot f_t (\|z\|^2) \cdot \sqrt{-1} \del \delbar f_t (\|z\|^2) \Big) \Bigg]
    \end{aligned}
    \end{equation}
    where $pr_t^{2,2}$ denotes the projection onto the $(2,2)$ component with respect to the complex structure $J_t$ on $X_t$ and $\rho_0$ and $\rho_t$ are smooth cutoff functions. In the above, $f_t$ is the function
    \begin{equation*}
        f_t(x) = \Big( \frac{|t|^2}{2} \Big)^{\frac{1}{3}} \int_0^{\cosh^{-1} (x/|t|)} (\sinh (2y) - 2y)^{\frac{1}{3}} dy. 
    \end{equation*}
These functions are from the Calabi--Yau local model $\omega_{\co,t} = \sqrt{-1} \partial \bar{\partial} f_t$, so that in a neighborhood $\{r < \delta \}$ where the cutoff functions $\rho \equiv 1$ there holds $\omega^2_t = \omega^2_{\co,t}$. We take a square-root of \eqref{fly-pre-per} to obtain the metric $\omega_t$.

\par $\bullet$ Step 2. The Fu--Li--Yau metrics correct $\omega_t$ by
\begin{equation} \label{fly-defn}
\omega_{\FLY,t}^2 = \omega_t^2 +  \partial \bar{\partial}^\dagger \partial^\dagger \gamma_t -  \bar{\partial} \partial^\dagger \bar{\partial}^\dagger \bar{\gamma_t}
\end{equation}
where $\gamma_t \in \Lambda^{2,3}(X_t)$ solves
\[
E_t(\gamma_t) = \bar{\partial} \omega_t^2, \quad d \gamma_t = 0
\]
and $E_t$ is the Kodaira--Spencer operator \cite{Kodaira}, which is a 4th order elliptic operator which acts on $(2,3)$-forms by 
\[
E_t = \partial \bar{\partial} \bar{\partial}^\dagger \partial^\dagger + \partial^\dagger \bar{\partial} \bar{\partial}^\dagger \partial + \partial^\dagger \partial .
\]
The adjoints here are with respect to $\omega_t$. The construction is such that $d \omega_{\FLY,t}^2 = 0$, and the main part of the argument in \cite{FLY12} is to prove that $\omega_{\FLY,t} > 0$ for small enough $t$.

We will need the following two properties of the Fu--Li--Yau metrics. These properties can be extracted from the estimates in \cite{FLY12}, and we refer to \cite{CPY21} for further discussion.

\begin{enumerate}[label = (FLY SM \Roman*), align = left]
    \item \label{prop-FLY-Ism} \textit{Local Model:} Near each singular point $s_i \in \mathcal{X}$, there
     exists constants $c,C,\epsilon,\delta>0$ and such that for all $t \in \Delta_\epsilon$, we have
\begin{equation}
\label{eqn-FLY-t-est}
\sup_{ \{ r \leq \delta \}} | g_{\FLY,t} - c \cdot g_{\co,t}|_{g_{\co,t}} \leq C |t|^{\frac{2}{3}}.
\end{equation}
Here $\Delta_\epsilon = \{ t \in \mathbb{C} : |t|< \epsilon \}$.
    \item \label{prop-FLY-IIsm} \textit{Uniform Convergence:} For any compact set $K \subset (X_0)_{\mathrm{reg}}$, the sequence $\Phi_t^* g_{\FLY,t}$ converges uniformly to $g_{\FLY,0}$ as $t \rightarrow 0$ on $K$.
\end{enumerate}

Our main theorem on the smoothings takes the form:

\begin{thm}
    \label{thm-FLY-SM-GH}
    The compact metric spaces $(X_t, d_{\FLY,t})$ converge to $(X_0,d_{\FLY,0})$ in the Gromov--Hausdorff topology as $t \to 0$.
\end{thm}

\subsubsection{Hermitian--Yang--Mills Metrics on the Smoothing}
\label{subsubsect-HYM-SM}

In order to get approximate Hermitian--Yang--Mills solutions on the smoothings, Collins--Picard--Yau glued the pullback of the metric $H_0$ to the Candelas--de la Ossa metrics \cite{CPY21}. These approximate metrics were perturbed to obtain true solutions $H_t$ to the Hermitian--Yang--Mills equations. The resulting metrics $H_t$ on $X_t$ solve
\begin{equation*}
    F_{H_t} \wedge \omega^2_{\FLY,t} =0, \quad   \int_{X_t} \log \frac{\det H_t}{\det g_{\FLY,t}} \, d {\rm vol}_{g_{\FLY,t}} = 0.
\end{equation*}
The construction of \cite{CPY21} is such that
\begin{equation*}
H_t = H_{t,{\rm ref}} e^{u}, \quad   | u |_{H_{t, {\rm ref}}} + r |\nabla u|_{H_{t, {\rm ref}}} \leq C |t|^{(1/3) |\beta|}
\end{equation*}
for $|\beta| \in (0,1)$, and
\begin{equation*}
    H_{t,{\rm ref}} = \chi g_{\co,t} + (1-\chi) [(\Phi_t^{-1})^* H_0]^{1,1}
\end{equation*}
and $\chi(z) = \zeta( |t|^{-\alpha} \| z \|^2 )$ for $0<\alpha<1$ and $\zeta: [0,\infty) \rightarrow [0,1]$ with $\zeta \equiv 1$ on $[0,1]$ and $\zeta \equiv 0$ on $[2,\infty)$. 

The metrics $H_t$ are uniformly equivalent to $g_{\FLY,t}$, so that
\begin{equation} \label{hym-sm-est}
    C^{-1} \cdot g_{\FLY,t} \leq H_t \leq C \cdot g_{\FLY,t}.
\end{equation}
Furthermore, for any compact set $K \subset (X_0)_{\mathrm{reg}}$, the sequence $\Phi_t^* H_t$ converges uniformly to $H_0$ as $t \rightarrow 0$ on $K$. We will show the Gromov--Hausdorff convergence of the Yang--Mills metrics.

\begin{thm}
    \label{thm-HYM-SM-GH}
    The compact metric spaces $(X_t, d_{H_t})$ converge to $(X_0, d_{H_0})$ in the Gromov--Hausdorff topology as $t \to 0$.
\end{thm}

\subsection{Notation and conventions}
\label{subsect-notn}

Before we continue with our study of conifold transitions, we establish a general notational guideline due to the sheer number of metrics involved:

When working with quantities related to metrics on small resolutions ($\what{X}$ and $\what{V}$), we will include a hat and a subscript to denote the metric being used. We will also use the parameters $a$ and $b$ for families of metrics on these spaces.

In a similar vein, analogous quantities on the smoothings ($X_t$ and $V_t$) and singular spaces ($X_0$ and $V_0$) will not have a hat, but will include an appropriate subscript. The parameters used for families of metrics here will be $s$ and $t$.

At times we will present lemmas and results that can be applied in more general settings encompassing both the small resolution and the smoothings. In this setting, we will not include the hat, but we will use the Greek letters $\alpha$ and $\beta$ as parameters.

For example:
\begin{center}
\begin{tabular}{| c | c |}
    \hline
    $\hat{g}_{\co,1}$ & Candelas--de la Ossa metric on the small resolution $\hat{V}$ at $a = 1$. \\
    \hline
    $g_{\co,1}$ & Candelas--de la Ossa metric on the smoothing $V_1$ at $t = 1$. \\
    \hline
    $\hat{d}_{\FLY,a}$ & Distance w.r.t the Fu--Li--Yau metric $\hat{g}_{\FLY,a}$ on the small resolution $\hat{X}$. \\
    \hline
    $d_{\FLY,t}$ & Distance w.r.t the Fu--Li--Yau metric ${g}_{\FLY,t}$ on the smoothing $X_t$. \\
    \hline
    $\hat{L}_{\hat{H}_a}(\gamma)$ & Length of a curve $\gamma$ w.r.t. the HYM metric $\hat{H}_a$ on the small resolution $\hat{X}$. \\
    \hline
    $L_{H_t}(\gamma)$ & Length of a curve $\gamma$ w.r.t. the HYM metric $H_t$ on the smoothing $X_t$. \\
    \hline
    $\diam_\alpha(S)$ & Diameter of a set $S$ w.r.t a metric $g_\alpha$ on a manifold $X$. \\
    \hline
\end{tabular}
\end{center}

In addition, we adopt the convention that $C$ denotes a generic positive constant that may change from line to line but do not depend on $a$ or $t$.

\section{Gromov--Hausdorff Continuity in the Regular Case}
\label{sect-GH-conv-reg}

In this section, we show Gromov--Hausdorff continuity of the geometries $(\what{X},\what{d}_{\FLY,a})$ and $(\what{X},\what{d}_{\what{H}_a})$ for $a > 0$, as well as $(X_t,d_{\FLY,t})$ and $(X_t,d_{H_t})$ for $t \neq 0$. That is, we show that the variations of geometry induced by a conifold transition is a continuous process in $\cal{M}$ away from the singular conifold. Namely, we have the following theorem:

\begin{thm}\label{thm-GH-cont-pos}
    The following paths $(0,1] \to \cal{M}$ are continuous in the Gromov--Hausdorff topology:
    \begin{enumerate}[label=(\roman*)]
        \item $a \mapsto (\what{X},\what{d}_{\FLY,a})$,
        \item $a \mapsto (\what{X},\what{d}_{H_a})$.
    \end{enumerate}
    Furthermore, the following maps $\Delta_\epsilon \setminus \{0\} \to \cal{M}$ are continuous in the Gromov--Hausdorff topology:
    \begin{enumerate}[label=(\roman*)]\addtocounter{enumi}{2}
            \item $t \mapsto (X_t,d_{\FLY,t})$,
        \item $t \mapsto (X_t,d_{H_t})$.
    \end{enumerate}
\end{thm}

Extending this continuity to the singular spaces obtained when $a=t=0$ is a more difficult task, which we will describe in Section~\ref{sect-GH-conv-sing}. Our final result will be that the maps $[0,1] \rightarrow \cal{M}$ and $\Delta_\epsilon \rightarrow \cal{M}$ are continuous, but in this current section we only consider the geometries away from $X_0$.

\subsection{Gromov--Hausdorff vs Uniform Convergence}

Since Riemannian manifolds exhibit more structure than that of a metric space, there is considerably more flexibility when defining notions of continuity of the geometry of a family of Riemannian manifolds than that of continuity in the Gromov--Hausdorff topology. In particular, a very natural way to define continuity of the geometry is through some continuity condition on a family of metrics. As the following well-known lemma shows, the Gromov--Hausdorff topology is weaker than the topology of uniform convergence of Riemannian metrics. Many similar results can be found in the literature, \cf Example 7.4.4 of \cite{BBI}, and we include the proof as a warm-up.

\begin{propn}
\label{propn-GH-convergence}
    Let $g_\alpha$ be a family of metrics on a connected, compact manifold $X$ of dimension $n$, where the parameter $\alpha \in U$ lies in an open set which we take to be either real or complex: $U \subset \mathbb{R}$ or $U \subset \mathbb{C}$. Fixing a parameter $\beta \in U$, suppose that $\alpha \mapsto g_{\alpha}$ is continuous at $\alpha = \beta$ in the $L^{\infty}$ norm with respect to $g_{\beta}$. Then $\alpha \mapsto (X,d_{g_{\alpha}})$ is continuous at $\alpha=\beta$ in the Gromov--Hausdorff topology.
\end{propn}

\begin{rmk}\label{rmk_L_inf}
    Since $X$ is compact, all metrics on $X$ are uniformly equivalent. That is, given two metrics $g,\tilde{g}$ on $X$, there is some $C > 1$ such that
    \[
    C^{-1} \cdot g < \tilde{g} < C \cdot g.
    \]
    Thus, the continuity assumption in Proposition~\ref{propn-GH-convergence} could be replaced by continuity of the family of metrics $g_{\alpha}$ in the $L^{\infty}$ norm with respect to any metric on $X$.
\end{rmk}

\begin{proof}
Let $0 < \epsilon < 1$ and fix a parameter $\beta$. Consider the identity map $(X,d_{g_{\alpha}}) \to (X,d_{g_{\beta}})$. This map is surjective, so it suffices to show that it is an $(C\eps)$-isometry when $|\alpha-\beta|$ is small, for some constant $C$ independent of $\alpha$.

The $L^\infty$ continuity of the metrics $g_\alpha$ at $\alpha=\beta$ implies that we may choose $\delta > 0$ sufficiently small so that if $|\alpha-\beta| < \delta$, then $\sup_{X} |g_{\alpha} - g_{\beta}|_{g_{\beta}} < \eps$. It follows that there exists some $\eps'=\eps'(\delta) \in \mathbb{R}$ such that for all $|\alpha-\beta|<\delta$, we have
\begin{equation*}
    (1-\eps') \cdot g_{\beta} \leq g_{\alpha} \leq (1+\eps') \cdot g_{\beta}.
\end{equation*}
Thus any $\alpha$ with $|\alpha-\beta|<\delta$, the length of a curve $\gamma$ satisfies $L_{\alpha}(\gamma) \leq (1+\eps') \cdot L_{\beta}(\gamma)$. It follows that $D := (1+\eps') \cdot \diam_\beta (X) \geq \diam_\alpha (X)$ for all $\alpha$ with $|\alpha-\beta|<\delta$. 

Pick points $p,q \in X$ and choose minimizing geodesics $\gamma_{\alpha}, \gamma_{\beta}: [0,1] \to X$ from $p$ to $q$ in the $g_\alpha$ and $g_{\beta}$ metrics, respectively. We have that $\gamma_{\alpha}(0) = \gamma_{\beta}(0) = p$ and $\gamma_{\alpha}(1) = \gamma_{\beta}(1) = q$, and furthermore $L_{\alpha}(\gamma_{\alpha}) = d_{\alpha} (p,q)$ and $L_{\beta}(\gamma_\beta) = d_{\beta}(p,q)$. Comparing the lengths of $\gamma_{\alpha}$ and $\gamma_{\beta}$ in the metrics $g_{\alpha}$ and $g_{\beta}$, we note that
\begin{align*}
\begin{split}
        |L_{\alpha}(\gamma_{\alpha}) - L_{\beta}(\gamma_{\alpha})| 
        & \leq \int_0^1 |g_{\alpha} - g_{\beta}|_{g_{\alpha}} \cdot|\dot{\gamma_{\alpha}}|_{g_{\alpha}} \, ds \\
        & \leq \left(\sup_{X} \, |g_{\alpha} - g_{\beta}|_{g_{\alpha}} \right) \cdot \int_0^1 |\dot{\gamma_{\alpha}}|_{g_{\alpha}}\,ds \\
        & < D \eps.
        \end{split}
    \end{align*}
Similarly, we have $|L_{\alpha}(\gamma_{\beta}) - L_{\beta}(\gamma_{\beta})| < D \eps$. Then we see that for $|{\alpha}-{\beta}| < \delta$, we have
    \begin{align*}
        d_{\alpha}(p,q) &\leq L_{\alpha}(\gamma_{\beta}) \leq L_\beta(\gamma_\beta) + D\eps \\
        &= d_{\beta}(p,q) + D \eps.
    \end{align*}
    Similarly, $d_{\beta}(p,q) < d_{\alpha}(p,q) + D\eps$, so that $|d_{\alpha}(p,q) - d_{\beta}(p,q)| < D\eps$ if $|{\alpha}-{\beta}| < \delta$. Since this choice of $\delta$ does not depend on the choice of $p,q \in X$, the identity map is a $D\eps$-isometry, completing the proof.
\end{proof}

When applying this statement to the geometry of the smoothings $X_t$, we will use the following variant:

\begin{cor} \label{cor-sm-GH}
    Let $\{ (X_{\alpha},g_\alpha) \}$ be a family of compact Riemannian manifolds parametrized by $\alpha \in U \subset \mathbb{C}$.  Fix $\beta \in U$, and suppose that for each $\alpha$ there is a diffeomorphism $F_\alpha: X_\beta \rightarrow X_\alpha$. Suppose $F_\alpha^* g_\alpha \rightarrow g_\beta$ in the $L^\infty$ norm with respect to $g_\beta$ as $\alpha \rightarrow \beta$. Then $(X_{\alpha},d_{g_{\alpha}}) \to (X_{\beta},d_{g_{\beta}})$ in the Gromov--Hausdorff topology as $\alpha \to \beta$.
\end{cor}

\begin{proof}
    This follows by applying the proposition above to the family of metrics $F_\alpha^* g_{\alpha}$ on the fixed manifold $X_{\beta}$. Thus $(X_{\beta}, F_\alpha^* g_{\alpha}) \to (X_{\beta}, g_{\beta})$ in the Gromov--Hausdorff topology as $\alpha \to \beta$. Since $(X_{\beta}, F_\alpha^* g_{\alpha})$ is isometric to $(X_{\alpha},g_{\alpha})$, the result follows.
\end{proof}

\subsection{Small Resolution Metrics \texorpdfstring{$\what{g}_{\FLY,a}$}{g}}
\label{subsect-GH-large-SR}

In order to prove Theorem~\ref{thm-GH-cont-pos}, it suffices by Proposition~\ref{propn-GH-convergence} to show that each of the families of metrics is continuous in the $L^{\infty}$ norm. To that end, we will show in this subsection that the family $\what{g}_{\FLY,a}$ is continuous in the $L^{\infty}$ norm.

\begin{lem}
\label{lem-FLY-SR}
    The Fu--Li--Yau metrics $\{\what{g}_{\FLY,a} \mid a \in (0,1]\}$ on $\what{X}$ satisfy the continuity condition of Proposition \ref{propn-GH-convergence}.
\end{lem}

\begin{proof}
    Let $b \in (0,1]$. Recall that the Fu--Li--Yau metrics are obtained via a gluing construction which interpolates between a multiple of the Candelas--de la Ossa metrics (near the $(-1,-1)$-curves) and the ambient Calabi--Yau metric (away from the $(-1,-1)$-curves) \cite{CPY21, FLY12}. The gluing region is independent of the parameter $a$, and $\what{\omega}_{\FLY,a}^2 - \what{\omega}_{\FLY,b}^2$ is supported on open sets around each $(-1,-1)$ curve, and in particular we have the following expression on the local models with $\| z \|^2 < 1$:
    \begin{equation}
    \label{eqn-omega^2-FLY-difference}
    \begin{aligned}
        \what{\omega}_{\FLY,a}^2 - \what{\omega}_{\FLY,b}^2 &= C\frac{2R^{-1}}{3} \sqrt{-1} \del \delbar \Big( \chi \Big( \frac{2R^2}{3} f_a(\|z\|^2) \Big) (\sqrt{-1} \del \delbar f_a (\|z\|^2) + 8a^2 \pi^* \omega_{FS}) \Big) \\
        &\qquad -  C\frac{2R^{-1}}{3} \sqrt{-1} \del \delbar \Big( \chi \Big( \frac{2R^2}{3} f_b(\|z\|^2) \Big) (\sqrt{-1} \del \delbar f_b (\|z\|^2) + 8b^2 \pi^* \omega_{FS}) \Big),
    \end{aligned}
    \end{equation}
    where $C$ and $R$ are constants, $\chi$ is a smooth function, and $f_s$ are a family of smooth functions such that $f_s(x) = s^2 f_1 (\frac{x}{s^3})$ and
    \begin{equation*}
        (x f_1')^3 + 6(x f_1')^2 = x^2.
    \end{equation*}
        (see \cite{CPY21}).
    It follows that the function $|\what{\omega}_{\FLY,a}^2 - \what{\omega}_{\FLY,b}^2|^2_{g_{\FLY,b}}$ is smooth in $a$ and $p$.

    Since $b \neq 0$, we can pick some $h > 0$ such that $I = [b-h,b+h] \subset (0,1]$ (or $I = [1-h,1] \subseteq (0,1]$ in the case where $b = 1)$. One can check that in coordinates around a point $p \in \what{X}$, each component in \eqref{eqn-omega^2-FLY-difference} is smooth in $a$ and $p$. In particular, differentiating the function $f_a (\|z\|^2)$ involves uniform bounds since we have the expression $f_a(\|z\|^2) = a^2 f_1 (\frac{\|z\|^2}{a^3})$ and also because $a > 0$ in our interval so that $\frac{\|z\|^2}{a^3}$ lies in a compact set.
        
    It follows that the covariant derivative of $|\what{\omega}_{\FLY,a}^2 - \what{\omega}_{\FLY,b}^2|^2_{g_{\FLY,b}}$ is continuous on $I \times \what{X}$. By compactness, we obtain uniform boundedness of the covariant derivative on $I$. By a corollary of the Arzel\`{a}--Ascoli theorem, the pointwise convergence of the function $|\what{\omega}_{\FLY,a}^2 - \what{\omega}_{\FLY,b}^2|^2_{g_{\FLY,b}}$ is actually uniform. 
    
    A positive $(n-1,n-1)$-form has a unique $(n-1)$-th root and this is determined in a continuous fashion (see \eg \cite{Michelsohn}). It follows that $\sup |\what{g}_{\co,a} - \what{g}_{\co,b}|_{\what{g}_{\co,b}}$ approaches $0$ as $a \rarr b$.
\end{proof}

We can now apply Proposition \ref{propn-GH-convergence} to this path of spaces to obtain the desired Gromov--Hausdorff continuity.

\subsection{Smoothing Metrics \texorpdfstring{$g_{\FLY,t}$}{g}}

We prove the analogous results of Section \ref{subsect-GH-large-SR} for the smoothings.

Let $\mu: \mathcal{X} \rightarrow \Delta$ be a holomorphic smoothing of $X_0$ and let $X_t = \mu^{-1}(t)$. Fix $s \neq 0$ and consider the smoothings $X_t$ nearby $X_s$. As this is a smooth family of complex manifolds, by Ehresmann's lemma there exists a smoothly varying family of diffeomorphisms $F_t: X_s \rightarrow X_t$ such that $F_s$ is the identity map. 

Recall that the Fu--Li--Yau \cite{FLY12} metrics on the smoothings are obtained by: 1) pullback and a gluing construction leading to a pre-perturbed metric $g_t$ followed by 2) a perturbation to a balanced metric $g_{\FLY,t}$.

    We see that the expression \eqref{fly-pre-per} for $\omega^2_t$ is smooth in the parameter $t$ and can employ the method in the proof of Lemma \ref{lem-FLY-SR}. The metric $g_t$ on $X_t$ is extracted from $\omega_t$ via a square root construction (see \eg \cite{Michelsohn}), and since the dependence on $t$ is explicit here and $F_t: X_s \rightarrow X_t$ varies smoothly with $F_s=\id$, we can see that
    \[
\lim_{t \rightarrow s} \sup_{X_s} | F_t^* g_t - g_s |_{g_s} = 0.
    \]
Corollary \ref{cor-sm-GH} applies to $(X_t,g_t)$, however these are not the Fu--Li--Yau metrics as these do not satisfy $d \omega_t^2=0$. For this we need to estimate the correction term $\gamma_t$ appearing in \eqref{fly-defn}. As this term $\gamma_t$ comes from solving $E_t(\gamma_t) = \bar{\partial} \omega_t^2$, we need to deduce from the fact that the right-hand sides $\bar{\partial} \omega_t^2$ vary smoothly for $t \neq 0$ that the solutions $\gamma_t$ vary smoothly. We first need to study some properties of the Kodaira--Spencer operator $E_t$ (which in this case is determined with respect to the auxiliary metric $\omega_t$).

\begin{lem}
Let $\what{X} \rightarrow X_0 \rightsquigarrow X_t$ be a conifold transition from an initial K\"ahler Calabi--Yau threefold with finite fundamental group. Endow $X_t$ with the auxiliary Hermitian metric $\omega_t$ from the Fu--Li--Yau construction. Then $E_t: \Lambda^{2,3}(X_t) \rightarrow \Lambda^{2,3}(X_t)$ satisfies $\ker E_t = \{ 0 \}$ for all $0<|t| \ll 1$.
\end{lem}

\begin{proof}
Let $\chi \in \Lambda^{2,3}(X_t)$ be such that $\chi \in \ker E_t$. Integrating by parts over the identity $\langle E_t \chi, \chi \rangle = 0$ implies
\[
\partial \chi = 0, \quad \bar{\partial}^\dagger \partial^\dagger \chi = 0.
\]
Next, we note that $\partial^\dagger \chi \in \Lambda^{1,3}(X_t)$ and so $\bar{\partial} (\partial^\dagger \chi) = 0$ by type consideration. It is noticed in \cite{FLY12} that
\[
H^{1,3}(X_t,\mathbb{C}) = H^0(X_t,T X_t) = 0
\]
by using $H^{1,3}(\what{X},\mathbb{C})=0$ on the small resolution together with Hartogs' lemma; we refer to \cite{FLY12} for the proof. Therefore
\[
\partial^\dagger \chi = \bar{\partial} \beta
\]
and so
\[
\langle \partial^\dagger \chi, \partial^\dagger \chi \rangle = \langle \bar{\partial}^\dagger \partial^\dagger \chi, \beta \rangle= 0.
\]
We conclude that if $\chi \in \Lambda^{2,3} \cap \ker E_t$, then $\partial \chi = \partial^\dagger \chi = 0$. It follows that $\psi = \bar{\chi} \in \Lambda^{3,2}(X_t)$ solves
\[
\Delta_{\bar{\partial}} \psi = 0, \quad \Delta_{\bar{\partial}} = \bar{\partial} \bar{\partial}^\dagger + \bar{\partial}^\dagger \bar{\partial}.
\]
By the Hodge theorem, this defines an element in Dolbeault cohomology, and since $X_t$ has trivial canonical bundle then $H^{3,2}(X_t,\mathbb{C}) = H^2(X_t,\Omega^3_{X_t}) = H^2(X_t,\mathcal{O}_{X_t})$. Lemma 8.2 in \cite{Fri91} states that if $H^2(\what{X},\mathcal{O}_{\what{X}}) = 0$ then $H^2(X_t,\mathcal{O}_{X_t}) =0$. Since 
\[
H^2(\what{X},\mathcal{O}_{\what{X}}) = H^{0,2}(\what{X},\mathbb{C})= H^{0,1}(\what{X},\mathbb{C})= 0
\]
on the initial K\"ahler Calabi--Yau threefold with finite fundamental group, we conclude that $H^{3,2}(X_t,\mathbb{C}) =0$ and so $\psi = 0$.
\end{proof}

We will also need some uniform estimates as $t \rightarrow s$. This is a standard argument given that the kernel of $E$ is trivial and $X_s$ is smooth.

\begin{lem} Fix $s>0$. There exists $\epsilon>0$ and $C>1$ such that
\begin{equation} \label{E-poincare}
\| \gamma_t \|_{C^{4,\alpha}(X_t)} \leq C \| E_t (\gamma_t) \|_{C^\alpha(X_t)} 
\end{equation}
for all $\gamma_t \in \Lambda^{2,3}(X_t)$ with $|t-s|<\epsilon$. Here each norm on $X_t$ is taken with respect to the auxiliary Hermitian metrics $\omega_t$ from the Fu--Li--Yau construction.
\end{lem}

\begin{proof}
Since the compact manifolds $X_t$ deform smoothly to the compact manifold $X_s$, the Schauder estimates
\begin{equation} \label{E-schauder}
\| \gamma_t \|_{C^{4,\alpha}(X_t)} \leq C (\| \gamma_t \|_{C^0(X_t)} + \| E_t (\gamma_t) \|_{C^\alpha(X_t)} )
\end{equation}
hold uniformly for all $t$ close to $s$ where the norms on $X_t$ are taken with respect to $\omega_t$. We would like to upgrade this estimate to \eqref{E-poincare}. 

Suppose \eqref{E-poincare} is false, so that there exists a sequence $t_i \rightarrow s$ and constants $C_i \rightarrow \infty$ with
\[
\| \gamma_{t_i} \|_{C^{4,\alpha}(X_{t_i})} \geq C_i \| E_{t_i} (\gamma_{t_i}) \|_{C^\alpha(X_{t_i})}. 
\]
Consider $\tilde{\gamma}_i = \gamma_{t_i} / \| \gamma_{t_i} \|_{C^{4,\alpha}(X_{t_i})}$. As
\[
\| \tilde{\gamma}_i \|_{C^{4,\alpha}(X_{t_i})} = 1, \quad \| E_{t_i} (\tilde{\gamma}_i) \|_{C^\alpha(X_{t_i})} \leq C_i^{-1}
\]
we may apply the Arzel\`{a}--Ascoli theorem to extract a convergent subsequence to a limit $\gamma_\infty$ solving
\[
E_s(\gamma_\infty)=0.
\]
By the previous lemma, $\gamma_\infty = 0$. This contradicts estimate \eqref{E-schauder}, which implies
\[
1 \leq C (\| \tilde{\gamma}_i \|_{C^0(X_{t_i})} + \| E_t (\tilde{\gamma}_i) \|_{C^\alpha(X_{t_i})} )
\]
and so $\frac{1}{2C} \leq \| \tilde{\gamma}_i \|_{C^0(X_{t_i})}$ for all $t_i$ close to $s$ and thus $\| \gamma_\infty \|_{C^0(X_s)} > 0$.
\end{proof}

Returning to the construction of the metrics $\omega_{\FLY,t}$, we claim $F_t^* \gamma_t \rightarrow \gamma_s$ in $C^4(X_s)$ as $t \rightarrow s$. Suppose not, so that there exists $\epsilon>0$ with
\begin{equation} \label{gamma-seq}
\| F_t^* \gamma_t - \gamma_s \|_{C^4(X_s)} \geq \epsilon
\end{equation}
along a subsequence $t_i \rightarrow s$. The uniform elliptic estimate \eqref{E-poincare} implies $\| \gamma_t \|_{C^{4,\alpha}(X_t)} \leq C$, and so $F_t^* \gamma_t$ is also bounded on $(X_s,g_s)$. Applying the Arzel\`{a}--Ascoli theorem, there is a subsequence converging to a limit $\gamma_\infty$ on $X_s$ solving
\[
E_s(\gamma_\infty) = \bar{\partial} \omega_s^2.
\]
It follows that
\[
E_s (\gamma_\infty-\gamma_s) = 0
\]
and since $\ker E_s = \{ 0 \}$, we conclude $\gamma_\infty = \gamma_s$, which contradicts \eqref{gamma-seq}.

Using that $F_t^* \gamma_t \rightarrow \gamma_s$, taking a square root of \eqref{fly-defn} gives a family of metrics $\omega_{\FLY,t}$ varying continuously as $t \rightarrow s$, and so
   \[
\lim_{t \rightarrow s} \sup_{X_s} | F_t^* g_{\FLY,t} - g_{\FLY,s} |_{g_s} = 0.
    \]
By Remark \ref{rmk_L_inf}, this convergence also holds with respect to the Fu--Li--Yau metrics $g_{\FLY,s}$ and thus Corollary \ref{cor-sm-GH} applies to $(X_t,g_{\FLY,t})$. This proves that $(X_t,d_{\FLY,t}) \rightarrow (X_s,d_{\FLY,s})$ in the Gromov--Hausdorff sense as $t \rightarrow s$.

\subsection{Small Resolution Metrics  \texorpdfstring{$\what{H}_a$}{H}}
We return to the small resolution $\what{X} \rightarrow X_0$, where there is a family of metrics $\what{H}_a$ satisfying the Hermitian--Yang--Mills equation
\[
F_{\what{H}_a} \wedge \what{\omega}_{\FLY,a}^2 = 0.
\]
We will show that for $b>0$ fixed, then 
\begin{equation} \label{HYM-a->b}
\lim_{a \rightarrow b}  \sup_{\what{X}}  \| \what{H}_a - \what{H}_b \|_{\what{H}_b} \rightarrow 0.
\end{equation}
Suppose this is false. Then there exists $\epsilon>0$ and a sequence $a_i \rightarrow b$ such that
\begin{equation*}
\| \what{H}_{a_i} - \what{H}_b \|_{\what{H}_b} \geq \epsilon, \quad \sqrt{-1} \Lambda_{\what{\omega}_{\FLY,a_i}} F_{\what{H}_{a_i}} = 0 
\end{equation*}
for all $a_i$. By the estimates in Proposition \ref{propn-HYM-ests-SR}, we have
\begin{equation*}
    C^{-1} \cdot \what{g}_{\FLY, b} \leq \what{H}_a \leq C \cdot \what{g}_{\FLY,b}.
\end{equation*}
Standard estimates for the Hermitian--Yang--Mills equations then give
\begin{equation} \label{hym-std-est}
|\nabla \what{H}_a|_{\what{g}_b} + |\nabla^2 \what{H}_a|_{\what{g}_b} \leq C.
\end{equation}
For a proof of these standard estimates, see \eg Proposition 3.9 with $r \equiv 1$ in \cite{CPY23} and the higher order estimates which follow after, or in the K\"ahler case Appendix C of \cite{jacobwalpuski}.
By the Arzel\`{a}--Ascoli theorem, we may extract a subsequence $\what{H}_{a_{i_k}}$ converging to a limit $H_\infty$ such that 
\begin{equation} \label{a->bcontra}
\| H_{\infty} - \what{H}_b \|_{\what{H}_b} \geq \epsilon, \quad \sqrt{-1} \Lambda_{\what{\omega}_{\FLY,b}} F_{H_\infty} = 0.
\end{equation}
This uses that $\what{\omega}_{\FLY, a_{i_k}} \rightarrow \what{\omega}_{\FLY,b}$ as $a_{i_k} \rightarrow b$, which holds by properties of the Fu--Li--Yau metrics. We now have two Hermitian--Yang--Mills metrics $H_\infty$ and $\what{H}_b$ with respect to $\what{\omega}_{\FLY, b}$. Consider $H_\infty= e^u \what{H}_b$. A classic calculation (see \eg \cite{uhlenbeckyau86}, or \eqref{HYM-int-est} below) implies
\begin{equation} \label{uniqueness-ab}
\Delta_{\what{\omega}_{\FLY, b}} |u|^2_{\what{H}_b} \geq C^{-1} |\bar{\partial} u|^2_{\what{\omega}_{\FLY, b},\what{H}_b}.
\end{equation}
It follows by the maximum principle that $\bar{\partial} u = 0$ and hence $u$ is a holomorphic endomorphism of $T^{1,0} \what{X}$. It is a well-known corollary of Yau’s theorem \cite{yau78} that on a Calabi-Yau threefold $\what{X}$ with finite fundamental group, there holds
\begin{equation} \label{no-endos}
H^0(\what{X},{\rm End} \, T^{1,0} \what{X}) = \{ c \, {\rm id} : c \in \mathbb{C} \}.
\end{equation}
Therefore $H_\infty = \lambda H_b$. By the normalization condition \eqref{hym-a-norma}, it follows that $\lambda=1$. This contradicts \eqref{a->bcontra}, and thus \eqref{HYM-a->b} is proved, and we conclude that $(\what{X},\what{d}_{\what{H}_a}) \rightarrow (\what{X},\what{d}_{\what{H}_b})$ in the Gromov--Hausdorff sense as $a \rightarrow b$.

Here is one way to see \eqref{no-endos}. By Yau’s theorem, $\what{X}$ admits a K\"ahler Ricci-flat metric. A holomorphic endomorphism of $T^{1,0} \what{X}$ is parallel with respect to the Levi-Civita connection by a Bochner argument. This parallel endomorphism defines a ${\rm Hol}_p$-linear map on $T_p \what{X}$, and hence by Schur’s lemma must be a scalar multiple of the identity at each point $p \in \what{X}$; this scalar factor is the same at all points by holomorphy. Here we used that the restricted holonomy group ${\rm Hol}_p = {\rm SU}(3)$ and acts irreducibly on $T_p \what{X}$ since $\what{X}$ has finite fundamental group by the de Rham splitting theorem.

\subsection{Smoothing Metrics  \texorpdfstring{$H_t$}{H}}
Fix $s \neq 0$ and consider the smoothings $X_t$ near the smooth fiber $X_s$ with smoothly varying family of diffeomorphisms $F_t: X_s \rightarrow X_t$ with $F_s$ the identity. The metrics $H_t$ also satisfy continuity of the form
\begin{equation}
    \lim_{t \rightarrow s} \| F_t^* H_t - H_s \|_{H_s} =0.
\end{equation}
The proof is similar to the arguments given before: suppose $F_t^* H_t$ does not converge to $H_s$ as $t \rightarrow s$ and extract a converging subsequence via the estimates \eqref{hym-sm-est} and \eqref{hym-std-est}. The limit solves the Hermitian--Yang--Mills equation, and by uniqueness and normalization then this limit must be $H_s$, which is a contradiction.

The main difference compared to the argument given above in \eqref{uniqueness-ab} is the step which establishes
\begin{equation} \label{no-endos2}
H^0(X_t,{\rm End} \, T^{1,0} X_t) = \{ c \, {\rm id} : c \in \mathbb{C} \},
\end{equation}
for all $t$ small enough. This follows from a well-known argument as soon as we show that $T^{1,0} X_t$ is a stable bundle: stable vector bundles are simple. Since we know that $T^{1,0} X_t$ admits a Hermitian-Yang-Mills metric, this bundle is polystable, and hence it remains to show that it cannot split holomorphically as $T^{1,0} X_t=A_t \oplus B_t$. Suppose by contradiction that there exists a sequence $t_i \rightarrow 0$ such that the tangent bundle splits, and after a subsequence and relabeling we may assume ${\rm rk} A_{t_i} = 1$. Define from such data a sequence of holomorphic endomorphisms $h_i \in H^0(X_{t_i}, {\rm End} \, T^{1,0} X_{t_i})$ by
\[
h_{t_i} =
\begin{cases}
  \mathrm{id} & \text{on } A_{t_i}, \\[4pt]
  0 & \text{on } B_{t_i}.
\end{cases}
\]
By compactness of holomorphic functions, on compact sets away from the singular set we may take a subsequential limit uniformly. By an exhaustion argument, a subsequence of the $t_i$ converges to a limiting endomorphism
\[
h_\infty \in H^0(X_{0,{\rm reg}}, {\rm End} \, T^{1,0} X_{0, {\rm reg}}).
\]
Taking pointwise limits of ${\rm Tr} \, h_{t_i}$ and $\det \, h_{t_i}$ gives
\begin{equation} \label{ti-limit-nosplit}
{\rm Tr} \, h_\infty = 1, \quad \det h_\infty = 0.
\end{equation}
This defines a holomorphic endomorphism of the tangent bundle on $\what{X} \backslash E$, and by Hartogs’ theorem this endomorphism extends to
\[
h_\infty \in H^0 (\what{X}, {\rm End} \, T^{1,0} \what{X}).
\]
We conclude by \eqref{no-endos} that $h_\infty = \lambda {\rm id}$, which contradicts \eqref{ti-limit-nosplit}.

\section{Gromov--Hausdorff Convergence in the Singular Case}
\label{sect-GH-conv-sing}

In this section, we extend Theorem~\ref{thm-GH-cont-pos}, and show that conifold transitions with the Fu--Li--Yau metrics and the Hermitian--Yang--Mills metrics are continuous in the Gromov--Hausdorff topology through the singular conifold at $t=a=0$. That is, we show that:

\begin{thm}
\label{thm-GH-sing}
The following four convergences hold in the Gromov--Hausdorff topology:
\begin{equation*}
\begin{matrix}
    \text{\underline{As $a \to 0$:}} & \text{\underline{As $t \to 0:$}}\\
    (\what{X}, \what{d}_{\FLY,a}) \rightarrow (X_0, d_{\FLY,0}), & (X_t, d_{\FLY,t}) \rightarrow (X_0, d_{\FLY,0})\\
    (\what{X},\what{d}_{\what{H}_a}) \rightarrow (X_0,d_{H_0}), & \quad (X_t,d_{H_t}) \rightarrow (X_0,d_{H_0}).
\end{matrix}
\end{equation*}
Therefore the maps $[0,1] \to \cal{M}$ given by
    \begin{enumerate}[label=(\roman*)]
        \item $a \mapsto (\what{X},\what{d}_{\FLY,a})$,
        \item $a \mapsto (\what{X},\what{d}_{H_a})$,
    \end{enumerate}
    and the maps $\Delta_\epsilon \to \cal{M}$ given by
    \begin{enumerate}[label=(\roman*)]\addtocounter{enumi}{2}
            \item $t \mapsto (X_t,d_{\FLY,t})$,
        \item $t \mapsto (X_t,d_{H_t})$,
    \end{enumerate}
    are continuous and agree at $a=t=0$.
\end{thm}

Before starting the proofs, we discuss how to interpret the limiting spaces $(X_0,d_{\FLY,0})$ and $(X_0,d_{H_0})$.

A \textbf{length structure} on a topological space $X$ consists of a class $A$ of curves in $X$, and a \emph{length} function $L:A \to [0,\infty)$. The curves in $A$ are called \textbf{admissible curves}. Both $A$ and $L$ have to satisfy a number of very natural conditions (see Definition~2.1.1 of \cite{BBI}). The length structure determines a distance function $d$ on $X$ so that for all $p,q \in X$, the distance $d(p,q)$ is the infimum of the lengths of all admissible curves from $p$ to $q$. A \textbf{length space} is a topological space $X$ together with a length structure on $X$.

\begin{rmk}
\label{rmk-metric-ext}

The spaces that we consider in this paper are all smooth Riemannian manifolds, with the exception of the cone $(V_0,g_{\mathrm{co},0})$ and the conifold $(X_0,g_{\mathrm{FLY},0})$ or $(X_0,H_0)$, where there are isolated singular points. Thus on each space considered, there is a natural length structure whose admissible curves are the piecewise differentiable curves, and whose length function is the standard Riemannian length -- \ie the integral of the speed of the curve with respect to the metric. The metric is defined everywhere except at the isolated singular points, so the Riemannian length is well defined.



\end{rmk}

We will make use of the following theorem (Theorem~2.5.23 of \cite{BBI}):
\begin{thm}
\label{thm-min-curve}
    Let $X$ be a complete, locally compact length space. Then given any $p,q \in X$, there exists an admissible curve $\gamma: [0,1] \to X$ such that $\gamma(0)=p$ and $\gamma(1)=q$, with $L(\gamma)=d(p,q)$.
\end{thm}

Lastly, we adopt the convention that the \textbf{diameter} of a set $Q$ is understood to mean the \emph{intrinsic diameter}, as we explain in the following definition:
\begin{defn}
\label{def-diameter}
    Let $Q$ be a bounded, path connected set in a length space $X$. Given two points $p,q \in Q$, the \textbf{intrinsic distance} from $p$ to $q$ is defined as $d_{\mathrm{int}}(p,q) := \inf L(\gamma)$, where the infimum is taken over all admissible curves $\gamma$ from $p$ to $q$ contained in $Q$. The \textbf{diameter} of $Q$ is defined by
\[
\diam(Q) := \sup\limits_{p,q \in Q} d_{\mathrm{int}}(p,q).
\]
Note that this is a non-standard definition of diameter, since many authors take the diameter of $Q$ to be the supremum of the distance (in $X$) between pairs of points in $Q$.
\end{defn}

\subsection{Reduction of a Curve}
\label{subsect-curve-reduction-SR}

In order to prove our main lemma (Lemma \ref{lem-SM-bigbox}), we will first need the following \emph{curve reduction lemma}:

\begin{lem}
\label{lem-reduced-curve}
Suppose $Q_1, \ldots, Q_k$ are disjoint, closed, path-connected, bounded sets in a complete, locally compact length space $X$, and let $\gamma: [0,1] \to X$ be an admissible curve. Then there exists an admissible curve $\mu:[0,1] \to X$ such that
\begin{enumerate}[label=(\roman*)]
    \item $\mu(0)=\gamma(0)$ and $\mu(1)=\gamma(1)$,
    \item For all $i \in \{1,\ldots,k\}$, the set $\mu^{-1}(Q_i) \subset [0,1]$ is either empty or a single closed subinterval of $[0,1]$, and:
    \item We have the estimate (noting Definition~\ref{def-diameter})
    \begin{equation*}
        L(\mu) \leq L(\gamma) + \sum\limits_{i=1}^k \diam(Q_i).
    \end{equation*}
\end{enumerate}
\end{lem}

\begin{proof}
    We will construct the curve $\mu$ in the following way:

    Define $a_1 \in [0,1]$ as $a_1 := \inf\left\{s \in [0,1] \mid \gamma(s) \in \cup_{i=1}^k Q_i\right\}$.
    Relabeling the sets $Q_i$ if necessary, we can say that $\gamma(a_1) \in Q_1$. Now, define a time $b_1 \in [0,1]$ by $b_1 := \sup\left\{s \in [0,1] \mid \gamma(s) \in Q_1\right\}$. Using Theorem~\ref{thm-min-curve}, take $\mu \mid_{[a_1,b_1]}$ to be any admissible curve such that $\mu\left([a_1,b_1]\right) \subset Q_1$, the endpoints satisfy $\mu(a_1)=\gamma(a_1)$ and $\mu(b_1)=\gamma(b_1)$, and furthermore
    \[
    L\left(\mu\mid_{[a_1,b_1]}\right) \leq \diam(Q_1).
    \]
     Now, for $i>1$ define $a_i$ to be $a_i := \inf\left\{s \in (b_{i-1},1] \mid \gamma(s) \in \cup_{i=1}^k Q_i\right\}$,
     and relabel the sets so that $\gamma(a_i) \in Q_i \neq Q_1, \ldots, Q_{i-1}$. Take $b_i$ to be the time $b_i := \sup \left\{s \in [0,1] \mid \gamma(s) \in Q_i\right\}$. Once again, choose $\mu \mid_{[a_i,b_i]}$ to be an admissible curve where $\mu\left([a_i,b_i]\right) \subset Q_i$, the endpoints are $\mu(a_i)=\gamma(a_i)$ and $\mu(b_i)=\gamma(b_i)$, and the length satisfies
     \[
    L\left(\mu\mid_{[a_i,b_i]}\right) \leq \diam(Q_i).
    \]
    Eventually, after $\ell \leq k$ iterations, there will not exist an $a_{\ell+1}$.

    At this point, we have constructed the curve $\mu$ on the set $A = \bigcup_{i=1}^{\ell} [a_i,b_i]$. For $s \in A':= [0,1]\, \setminus\,A$, set $\mu(s)=\gamma(s)$.

    Since the class of admissible curves is closed under restrictions and concatenations (see Definition 2.1.1 of \cite{BBI}), we see by construction that $\mu$ is admissible. Furthermore $\mu^{-1}(Q_i)=[a_i,b_i]$ for $1 \leq i \leq \ell$, and $\mu^{-1}(Q_i)=\varnothing$ otherwise. Finally, note that
    \begin{align*}
        L(\mu) &= L\left(\mu \mid_{A'} \right) + \sum\limits_{i=1}^{\ell} L\left(\mu \mid_{[a_i,b_i]}\right) \leq L\left(\gamma\mid_{A'}\right) + \sum\limits_{i=1}^{\ell} \diam(Q_i)\\
        &\leq L(\gamma) + \sum\limits_{i=1}^k \diam(Q_i),
 \end{align*}
 completing the proof.
\end{proof}

\subsection{The Main Lemma}

Gromov--Hausdorff convergence of the various metrics on both the small resolution and the smoothing will follow by applying the following general lemma. A similar strategy is used in \cite{SW13}. With this lemma in place, it will remain to verify its hypothesis in our geometric setups.

\begin{lem}
\label{lem-SM-bigbox}
    Let $X_{\alpha}$ be a family of connected compact smooth manifolds where the parameter $\alpha$ lies in either $\alpha \in (0,1]$ or $\alpha \in \Delta \backslash \{ 0 \} \subset \mathbb{C}$. Let $X_0$ be a compact analytic space with $X_0 = (X_0)_{\rm reg} \cup (X_0)_{\rm sing}$ where $(X_0)_{\rm reg}$ is a connected smooth manifold and there are finitely many ODP singular points $(X_0)_{\rm sing} = \{s_1, \ldots, s_k\}$, meaning that each $s_i \in X_0$ is contained in a neighborhood $U_i \subset X_0$ which can be identified with a neighborhood of the origin in $V_0 \subset \mathbb{C}^4$. 
    
    For each $\alpha$, let $K_{i,\alpha} \subseteq X_0$ and $C_{i,\alpha} \subseteq X_\alpha$ be disjoint compact sets with $s_i \in K_{i,\alpha}$, for $i \in \{1,\ldots,k\}$. Suppose further that we have a family of maps $F_{\alpha} \: X_{\alpha} \rarr X_0$ such that
    \begin{itemize}
        \item The restriction $F_{\alpha} \: X_{\alpha} \bs \bigcup_i C_{i,\alpha} \to X_0 \bs \bigcup_i K_{i,\alpha}$ is a diffeomorphism, and
        \item For each $i \in \{1,\ldots,k\}$, we have $F_{\alpha}(C_{i,\alpha}) \subset K_{i,\alpha}$.
    \end{itemize} 
    
  Let $g_\alpha$ be a Riemannian metric on $X_\alpha$ for each $\alpha$. Let $g_0$ be a smooth Riemannian metric on $(X_0)_{\rm reg}$ satisfying the bound $g_0 \leq C (dr^2 + r^2 \cdot g_L)$ in a neighborhood $U_i$ of the singular points $s_i$. Let $d_0$ be the distance function induced by $g_0$ on $X_0$ (see Remark \ref{rmk-metric-ext}).
    
    Now, let $\epsilon >0$, and suppose that there exist disjoint open sets $G_1, \ldots, G_k \subset X_0$ and $\alpha_0 > 0$ such that each $G_i$  satisfies
    \begin{enumerate}[label=(\roman*)]
        \item $K_{i,\alpha} \subset G_i$ for all $|\alpha|<\alpha_0$,

        \item $(F_\alpha^{-1})^* g_\alpha$ converges uniformly to $g_0$ on the compact set $X_0 \backslash \bigcup_i G_i$ as $\alpha \to 0$,
    
        \item $\diam_0 (G_i) < \epsilon$, and
        
        \item $\diam_{\alpha} (F_{\alpha}^{-1}(G_i)) < \epsilon$ whenever $|\alpha| < \alpha_0$.
    \end{enumerate}

    Then there exists $\alpha_1 > 0$ and a constant $C > 0$ independent of $\alpha$ such that 
    \begin{equation*}
        F_{\alpha} \: (X_{\alpha}, d_{\alpha}) \rarr (X_0, d_0)
    \end{equation*}
    is a $C\eps$-isometry for all $|\alpha| < \alpha_1$. 
\end{lem}

\begin{proof}
    Let $\eps > 0$. We first prove that the image of each $F_{\alpha}$ is $\eps$-dense in $X_0$. By our assumptions, the only points in $X_0$ not in $F_{\alpha}(X_{\alpha})$ must lie in some $K_{i,{\alpha}}$. For each $i$, we can choose some $p \in \overline{G_i} \bs K_{i,\alpha}$ which is in the image of $F_{\alpha}$.  Since $\diam_{g_0} (\overline{G_i}) < \epsilon$, we have that $F_{\alpha}(X_{\alpha})$ is $\eps$-dense in $X_0$ with respect to $d_0$ for sufficiently small $\alpha$.

    It remains to prove that there exists some $C,\alpha_1>0$ such that for all $|\alpha|<\alpha_1$ then
    \begin{equation} \label{mainlemineq}
        |d_{\alpha}(p,q) - d_0(F_{\alpha}(p),F_{\alpha}(q))| < C\eps
    \end{equation}
    for each $p,q \in X_{\alpha}$. 

    Let $p,q \in X_{\alpha}$. Using Theorem~\ref{thm-min-curve}, pick a curve $\gamma \: [0,1] \rarr X_0$ such that $\gamma(0) = F_{\alpha}(p)$ and $\gamma(1) = F_{\alpha}(q)$ and
    \begin{equation} \label{step1}
        L_0(\gamma) = d_0(F_{\alpha}(p),F_{\alpha}(q)).
    \end{equation}
We will replace this curve $\gamma$ with a curve $\mu$ on $X_0$ passing through the bad sets $\overline{G_i}$ at most $k$ times using Lemma \ref{lem-reduced-curve}. The new curve $\mu$ is piecewise differentiable with $\mu(0) = F_{\alpha}(p)$, $\mu(1) = F_{\alpha}(q)$,
    \begin{equation} \label{step2}
        L_0(\mu) \leq L_0(\gamma) + \sum_{i=1}^k \diam_0(\overline{G_i}) \leq L_0(\gamma) + k \eps,
    \end{equation}
and the construction of Lemma \ref{lem-reduced-curve} provides an integer $\ell \leq k$ and a sequence
    \begin{equation*}
        0 \leq a_1 \leq b_1 < \ldots < a_\ell \leq b_\ell \leq 1,
    \end{equation*}
    such that (by relabelling $s_i$ if necessary) we have $\mu^{-1}(\overline{G_i}) = [a_i,b_i]$ for $1 \leq i \leq \ell$ and $\mu^{-1}(\overline{G_i}) = \emptyset$ for $\ell + 1 \leq i \leq k$. Set $A_i = [a_i,b_i]$ and $A' = [0,1]\, \bs\, \bigcup_{i=1}^\ell A_i$.

   Over the closed time intervals $\overline{A'}$, the curve $\mu$ does not enter any $K_{i,\alpha}$, and can be identified with a curve on $X_\alpha$ by the diffeomorphism $F_\alpha$. Define a curve $\mu_{\alpha}: \overline{A'} \rightarrow X_\alpha$ on $X_{\alpha}$ by $\mu_{\alpha}(s) = F_{\alpha}^{-1} \circ \mu(s)$.
   

    
    By the triangle inequality and the diameter estimate $\diam_{\alpha} (F_{\alpha}^{-1}(\overline{G_i})) < \eps$, we have that
    \begin{equation} 
    \begin{aligned}
        d_{\alpha}(p,q) &\leq d(p,\mu_{\alpha}(a_1))+\sum\limits_{i=1}^{\ell} d_{\alpha}(\mu_{\alpha}(a_i),\mu_{\alpha}(b_i)) + \sum\limits_{i=2}^{\ell} d_{\alpha}(\mu_{\alpha}(b_{i-1}),\mu_{\alpha}(a_i)) + d_{\alpha}(\mu_{\alpha}(b_{\ell},q)\\
        &\leq L_{\alpha}({\mu}_{\alpha}|_{[0,a_1]}) + \sum\limits_{i=2}^{\ell} L_{\alpha}(\mu_{\alpha} \mid_{[b_{i-1},a_i]}) + L_{\alpha}(\mu_{\alpha} \mid_{[b_{\ell},1]})+k\eps \\
        &\leq \int_{A'} |\dot{\mu}_{\alpha}(s)|_{g_{\alpha}}\, ds + k \eps = \int_{A'} |(F_{\alpha}^{-1})_* \dot{\mu}(s) |_{g_{\alpha}}\, ds + k \eps. \label{dpq1}
    \end{aligned}
    \end{equation}
    The set $A'$ is defined such that $\mu|_{A'} \in X_0 \bs \bigcup_i G_i$. The uniform convergence of the metrics $(F_{\alpha}^{-1})^*g_{\alpha}$ to $g_0$ on this region gives that
    \begin{equation*} 
        \int_{A'} |(F_{\alpha}^{-1})_* \dot{\mu}(s) |_{g_{\alpha}}\, ds \leq (1+ \delta) \int_{A'} |\dot{\mu}(s) |_{g_0}\, ds \leq (1+\delta) L_0(\mu),
    \end{equation*}
    and $\delta$ can be made arbitrarily small for sufficiently small $\alpha$. We can next apply \eqref{step1}, \eqref{step2} to obtain 
      \begin{equation} \label{dpq2}
        \int_{A'} |(F_{\alpha}^{-1})_* \dot{\mu}(s) |_{g_{\alpha}}\, ds \leq L_0(\mu) + \delta \, (\diam_{0}(X_0)+k \epsilon).
    \end{equation}
     Note that $\diam_{0}(X_0)< \infty$ since it is a union of a smooth geometry on a compact manifold $X_0 \bs \bigcup_i G_i$ with sets $\overline{G_i}$ of bounded diameter that have non-trivial intersection with $X_0 \bs \bigcup_i G_i$. Combining \eqref{dpq1} and \eqref{dpq2} and choosing $\delta$ small enough gives
    \begin{equation*}
        d_\alpha(p,q) \leq L_0(\mu) + (k+1) \epsilon.
    \end{equation*}
    Applying \eqref{step1} and \eqref{step2}, we then have
    \begin{equation} \label{mainlemineq2}
        d_{\alpha}(p,q) \leq d_0(F(p),F(q)) + (2k+1) \eps.
    \end{equation}
    We now need to obtain the other side of the desired inequality \eqref{mainlemineq}, and the argument is similar. Let $\eta_{\alpha} \: [0,1] \rarr X_{\alpha}$ be a curve such that $\eta_{\alpha}(0) = p$ and $\eta_{\alpha}(1) = q$, and
    \begin{equation*}
        L_{\alpha}(\eta_{\alpha}) = d_{\alpha}(p,q).
    \end{equation*}
   As before, we use Lemma \ref{lem-reduced-curve} to replace $\eta_\alpha$ with a curve $\nu_\alpha$ passing through the bad sets $F_\alpha^{-1}(\overline{G_i})$ at most $k$-times. The replacement curve $\eta_\alpha:[0,1] \rightarrow X_\alpha$ has the same endpoints $\nu_{\alpha}$ with $\nu_{\alpha}(0) = p$, $\nu_{\alpha}(1) = q$ and satisfies the length estimate
    \begin{equation} \label{FpFq3}
        L_{\alpha}(\nu_{\alpha}) \leq L_{\alpha}(\eta_{\alpha}) + \sum_{i=1}^k \diam_{\alpha}(F_{\alpha}^{-1}(\overline{G_i})) \leq L_{\alpha}(\eta_{\alpha}) + k \eps.
    \end{equation}
The time interval can be broken into $[0,1] = A_{\alpha} \cup A'_{\alpha}$ as before where $\nu_\alpha|_{A'_\alpha} \in X_\alpha \backslash \bigcup F_\alpha^{-1}(G_i)$.
    

   We now move onto the space $(X_0,d_0)$. Define a curve $\nu: \overline{A'_\alpha} \rightarrow X_0$ by $\nu(s) = F_{\alpha} \circ \nu_{\alpha}(s)$, and apply the triangle inequality as in \eqref{dpq1} to obtain
    \begin{equation} \label{FpFq1}
        d_0(F(p),F(q)) \leq \int_{A'_\alpha} |\dot{\nu}|_{g_0} ds + k \epsilon = \int_{A'_\alpha} |\dot{\nu}_\alpha |_{F_\alpha^* g_0} ds + k \epsilon.
    \end{equation}
   

The convergence of the metrics $(F_{\alpha}^{-1})^* g_{\alpha}$ to $g_0$ on $X_0 \backslash \bigcup_i G_i$ and the fact that $\nu_\alpha|_{A'_\alpha}$ stays within $X_\alpha \backslash \bigcup F_\alpha^{-1}(G_i)$ implies that
    \begin{equation} \label{FpFq2}
        \int_{A'_\alpha} | \dot{\nu}_{\alpha}(s) |_{F_\alpha^* g_0}\, ds \leq  \int_{A'_\alpha} (1+|F_\alpha^* g_0 - g_\alpha|_{g_\alpha} )|\dot{\nu}_\alpha|_{g_\alpha} ds \leq L_\alpha(\nu_\alpha) + \delta \, (\diam_{\alpha}(X_\alpha)+1),
    \end{equation}
   where $\delta$ is small for sufficiently small $\alpha$. We can bound uniformly in $\alpha$ the diameter
   \begin{equation*}
       \diam_{\alpha}(X_\alpha) \leq C.
   \end{equation*}
   For this, note that $(X_\alpha \backslash \bigcup_i F_\alpha^{-1}(G_i), g_\alpha)$ is isometric to $(X_0 \backslash \bigcup_i G_i, (F_\alpha^{-1})^*g_\alpha)$ which has bounded diameter since $(F_\alpha^{-1})^* g_\alpha \rightarrow g_0$ smoothly uniformly on this region. The remaining piece of the geometry $(X_\alpha,g_\alpha)$, namely the sets $\bigcup_i F_\alpha^{-1}(\overline{G_i})$, also have bounded diameter.

  From here, we can combine \eqref{FpFq1}, \eqref{FpFq2} and \eqref{FpFq3} and choose $\delta$ small enough to establish
\begin{equation*}
    d_0(F(p),F(q)) \leq d_\alpha(p,q) + (2k+1) \epsilon.
\end{equation*}
    Combining this together with \eqref{mainlemineq2}, we obtain \eqref{mainlemineq} and the lemma holds for the uniform constant $C = 2k+1$.
\end{proof}

\subsection{Estimates on the Small Resolution}
\label{subsect-ests-SR}

In this subsection, we will show how Lemma~\ref{lem-SM-bigbox} gives convergence of the families of metrics on the small resolution. In the small resolution case, the maps $F_\alpha$ are simply the blowdown map $F: \what{X} \to X_0$, while the sets $C_{i,\alpha} \subset \what{X}$ are the $(-1,-1)$-curves $E_i \simeq \mathbb{P}^1$, and the sets $K_{i,\alpha} \subset X_0$ are the singletons $K_{i,\alpha} = \{s_i\}$ containing the conifold singularities. At this point, we must check that the diameter estimates (ii) and (iii) appearing in Lemma~\ref{lem-SM-bigbox} apply for the small resolution metrics. Since these are local estimates around the $(-1,-1)$-curves and around the singularities, we work on the local model $(\what{V},\what{g}_{\co,a})$. In order to get a handle on bounds pertaining to the ``tube" $\what{T}(1) = \{ r \leq 1 \}$, we split it up into a smaller ``tube" $\what{T}(aK)$ and an ``annulus" $\what{T}(1)\, \bs\, \what{T}(aK)$.

\subsubsection{Tubular Bounds}
\label{subsubsect-tubular-SR}

Recall the Asymptotically Conical Decay Property \ref{prop-SR-II}. We may fix a constant $K$ such that
\begin{equation*}
    |(\pi^{-1})^* (\what{g}_{\co,a}) - g_{\co,0}|_{g_{\co,0}} \leq \frac{1}{2} 
\end{equation*}
when $r > aK$. We start with uniform bounds on the spaces $(\what{T}(aK), \what{g}_{\co,a})$. These will be obtained using the Scaling Property \ref{prop-SR-I} and the compactness of the set $(\what{T}(K), \what{g}_{\co,1})$.

To estimate the diameter, we consider a curve $\gamma \: [0,1] \rarr \what{T}(aK)$. The length of this curve with respect to the metric $\what{g}_{\co,a}$ is given by
\begin{equation*}
\begin{aligned}
    \what{L}_{\co,a} (\gamma) &= \int_0^1 \sqrt{\what{g}_{\co,a} \Big( \dot{\gamma}(s), \dot{\gamma}(s) \Big)} \, ds \\
    &= \int_0^1 \sqrt{ a^2 \cdot S^*_{a^{-1}} (\what{g}_1) \Big( \dot{\gamma}(s), \dot{\gamma}(s) \Big)} \, ds \\
    &= a \cdot \int_0^1 \sqrt{ \what{g}_{\co,1} \Big( (S_{a^{-1}})_* \dot{\gamma}(s), (S_{a^{-1}})_* \dot{\gamma}(s) \Big)} \, ds \\
    &= a \cdot \what{L}_{\co,1} ( S_{a^{-1}} \circ \gamma ),
\end{aligned}
\end{equation*}
where we have used \ref{prop-SR-I}. Since there is a one-to-one correspondence between curves in $\what{T}(aK)$ and curves in $\what{T}(K)$ given by composition with $S_{a^{-1}}$, it follows that
\begin{equation}
\label{eqn-tubular-diam-a-bound}
    \what{\diam}_{\co,a} (\what{T}(aK)) = a \cdot \what{\diam}_{\co,1} (\what{T}(K)).
\end{equation}

To obtain a volume bound, we note that
\begin{equation*}
    \what{\Vol}_{\co,a} (\what{T}(aK)) = \int_{\what{T}(aK)} \what{\omega}_{\co,a}^3 = \int_{\what{T}(aK)} a^6 \cdot S^*_{a^{-1}} (\what{\omega}_{\co,1}^3).
\end{equation*}
Using the change of variables formula, this becomes
\begin{equation*}
    \what{\Vol}_{\co,a} (\what{T}(aK)) = a^6 \cdot \int_{S_{a^{-1}}(\what{T}(aK))} \what{\omega}_{\co,1}^3 = a^6 \cdot \int_{\what{T}(K)} \what{\omega}_{\co,1}^3 .
\end{equation*}
Therefore
\begin{equation}
\label{eqn-tubular-Vol_{co,a}-bound}
    \what{\Vol}_{\co,a} (\what{T}(aK)) = a^6 \cdot \what{\Vol}_{\co,1} (\what{T}(K)).
\end{equation}

\subsubsection{Annular Bounds}
\label{subsubsect-annular-SR}

Let $\delta > 0$. We will obtain diameter and volume bounds on the annular region $\what{T}(\delta) \bs \what{T}(aK)$ for $0 < a \leq \frac{\delta}{K}$. These are derived using the Asymptotically Conical Decay Property \ref{prop-SR-II}.

Fix a point $p = (\lambda, u_0, v_0) \in \what{T}(\delta) \bs \what{T}(aK)$, and denote $\rho = r(p)$. Then $\rho \in (aK,\delta]$. 

Consider the curve $\what{\gamma}: [\frac{a}{\rho}K, 1] \rightarrow \what{T}(\delta)$ given by
\begin{equation}
    \what{\gamma}(s) = (\lambda_0, s^{\frac{3}{2}} u_0, s^{\frac{3}{2}} v_0).
\end{equation}
This path begins in $\what{T}(aK)$ and moves along the fiber over $\lambda_0$ to arrive at $p=\what{\gamma}(1)$. 

Using the blowdown map $\pi: \what{V} \rightarrow V_0$ \eqref{pi-defn}, it can be directly checked that this curve is sent to the  curve $\gamma = \pi \circ \what{\gamma}$ in $V_0$ given by:
\begin{equation*}
    \gamma(s) = s^{\frac{3}{2}} \cdot \pi(\lambda_0, u_0, v_0).
\end{equation*}
It follows that
\begin{equation*}
    r (\gamma (s)) = s \cdot \rho.
\end{equation*}

\begin{lem} \label{lem-cone-ball}
The path $\gamma(s)$ on $V_0$ given above has speed $|\dot{\gamma}|_{g_{\co,0}} = \rho$, and length $L_{\co,0}(\gamma)=\rho-aK$.
\end{lem}

\begin{proof}
The cone metric can be written as $g_{\co,0} = dr^2 + r^2 \cdot {\rm pr}_1^* g_L$, where $g_L$ is a metric on the link $L = \{ r = 1\}$ and ${\rm pr}_1: V_0 \rightarrow L$ is the projection to the link ${\rm pr}_1(z) = \frac{z}{\| z \|}$. We can then compute
\begin{equation*}
dr (\dot{\gamma}) = \frac{d}{ds} (r \circ \gamma) = \rho,
\end{equation*}
and
\begin{equation*}
    ({\rm pr}_1)_* \dot{\gamma} = \frac{d}{ds} ({\rm pr}_1 \circ \gamma) = 0.
\end{equation*}
Therefore
\begin{equation*}
    g_{\co,0}(\dot{\gamma},\dot{\gamma}) = \rho^2.
\end{equation*}
This gives the speed of $\gamma$, and integration gives the length.
\end{proof}

We now compare the length of the curve $(\what{\gamma}, \what{g}_{\co,a})$ to the length of the curve $(\gamma, g_{\co,0})$.
\begin{equation*}
\begin{aligned}
    |\what{L}_{\co,a}(\what{\gamma}) - L_{\co,0} (\gamma)|
    &= \Bigg| \int_{\frac{a}{\rho}K}^1 |\dot{\what{\gamma}}|_{\what{g}_{\co,a}} - |\dot{\gamma}|_{g_{\co,0}} \, ds \Bigg| \\
    &\leq \int_{\frac{a}{\rho}K}^1 \bigg| |\dot{\what{\gamma}}|_{\what{g}_{\co,a}} - |\dot{\gamma}|_{g_{\co,0}} \bigg| \, ds \\
    &= \int_{\frac{a}{\rho}K}^1 \bigg| |\dot{\gamma}|_{(\pi^{-1})* (\what{g}_{\co,a})} - |\dot{\gamma}|_{g_{\co,0}} \bigg| \, ds \\
    &= \int_{\frac{a}{\rho}K}^1 \Bigg| \frac{|\dot{\gamma}|^2_{(\pi^{-1})* (\what{g}_{\co,a})} - |\dot{\gamma}|^2_{g_{\co,0}}}{|\dot{\gamma}|_{(\pi^{-1})* (\what{g}_{\co,a})} + |\dot{\gamma}|_{g_{\co,0}}} \Bigg| \, ds.
\end{aligned}
\end{equation*}
We then obtain the estimate
\begin{equation*}
 \begin{aligned}
|\what{L}_{\co,a}(\what{\gamma}) - L_{\co,0} (\gamma)| &\leq \int_{\frac{a}{\rho}K}^1 \frac{| (\pi^{-1})^* \what{g}_{\co,a} - g_{\co,0}|_{g_{\co,0}} \, |\dot{\gamma}|_{g_{\co,0}}^2}{|\dot{\gamma}|_{(\pi^{-1})* (\what{g}_{\co,a})} + |\dot{\gamma}|_{g_{\co,0}}} \, ds \\
&\leq \int_{\frac{a}{\rho}K}^1 | (\pi^{-1})^* \what{g}_{\co,a} - g_{\co,0}|_{g_{\co,0}} \, |\dot{\gamma}|_{g_{\co,0}} \, ds.
 \end{aligned}   
\end{equation*}
We now use $|\dot{\gamma}|_{g_{\co,0}}=\rho$, $r(\gamma(s)) = s \cdot \rho$, and \ref{prop-SR-II} to obtain
\begin{equation*}
|\what{L}_{\co,a}(\what{\gamma}) - L_{\co,0} (\gamma)| 
\leq \int_{\frac{a}{\rho}K}^1 \frac{C a^2 \rho}{r^2} \, ds = \frac{Ca^2}{\rho} \cdot  \int_{\frac{a}{\rho}K}^1 \frac{1}{s^2} \, ds = Ca \cdot \Big( \frac{1}{K} - \frac{a}{\rho} \Big).
\end{equation*}
Therefore
\begin{equation*}
    \what{L}_{\co,a}(\what{\gamma}) \leq |\what{L}_{\co,a}(\what{\gamma}) - L_{\co,0} (\gamma)| + |L_{\co,0}(\gamma)| \leq Ca \cdot \Big( \frac{1}{K} - \frac{a}{\rho} \Big) + (\rho - aK).
\end{equation*}
For fixed $0 < a \leq \frac{\delta}{K}$, this is maximized when $\rho = \delta$, giving
\begin{equation*}
    \what{L}_{\co,a}(\what{\gamma}) \leq  C \cdot \left(\frac{a}{K\delta} + 1\right) \cdot (\delta - aK) \leq C \cdot (\delta - aK)
\end{equation*}
As such, we get
\begin{equation*}
    \what{d}_{\co,a} (p, \what{T}(aK)) \leq  C \cdot (\delta - aK).
\end{equation*}

In tandem with our diameter bound for $\what{T}(aK)$ \eqref{eqn-tubular-diam-a-bound}, we get
\begin{equation*}
    \what{\diam}_{\co,a} (\what{T}(\delta)) \leq a \cdot \what{\diam}_{\co,1} (\what{T}(K)) + 2C \cdot  (\delta - aK),
\end{equation*}
which for fixed $\delta$ and $K$, is uniformly bounded for $0 < a \leq \frac{\delta}{K}$.

In the appendix, we will need a volume estimate on tubes around the exceptional divisor. This volume estimate follows from the diameter estimate and the Bishop--Gromov Comparison Theorem. Indeed, for $0 < a \leq \frac{\delta}{K}$, the diameter estimate tells us that
\[
\what{\diam}_{\co,a} (\what{T}(\delta)) \leq C \delta.
\]
Thus the tube $\what{T}(\delta)$ is contained in the ball $\what{B}(p,C\delta)$ for any point $p \in \what{T}(\delta)$. Therefore
\begin{align*}
    \what{\Vol}_{\co,a}(\what{T}(\delta)) &\leq \what{\Vol} \, \what{B}(p,C\delta) \\
    &\leq \Vol_{\mathrm{Euc}} \,B(0,C\delta),
\end{align*}
where $\Vol_{\mathrm{Euc}} \, B(0,C\delta)$ denotes the Euclidean volume of the ball of radius $C\delta$ in $\mathbb{R}^6$. This second inequality is by Bishop--Gromov, together with the fact that the metrics $\what{g}_{\co,a}$ are Ricci-flat. Thus, $ \what{\Vol}_{\co,a}(\what{T}(\delta)) \leq C \delta^6$.
We record these diameter and volume bounds for future reference. The diameter bound is used in the next section, and the volume bound is used in the appendix.
\begin{lem}
\label{lem-vol-diam-a-bounds}
    For $\delta > 0$, we have
    \begin{equation}
    \label{eqn-diam-a-bound}
        \what{\diam}_{\co,a} (\what{T}(\delta)) \leq C \delta,
    \end{equation}
    and
        \begin{equation}
    \label{eqn-vol-a-bound}
        \what{{\rm Vol}}_{\co,a}(\hat{T}(\delta)) \leq C \delta^6.
 \end{equation}
    for any $0 < a \leq \frac{\delta}{K}$.
\end{lem}

\subsubsection{Applying the Main Lemma}

Our diameter estimates will enable us to prove the following useful lemma akin to that of Song--Weinkove \cite{SW13}:

\begin{lem}
\label{lem-suff-small-discs-SR-model}
For $0< \eps < 1$, there exists $\delta > 0$ and $0 < a_0$ such that for $0 < a < a_0$
\begin{enumerate}[label = \roman*)]

    \item $\diam_{\co,0} (D_0(\delta)) < \eps$, and

    \item $\what{\diam}_{\co,a} (\pi^{-1} (D_0(\delta))) < \eps$.
\end{enumerate}
\end{lem}

\begin{proof}
    We have that $D_0(\delta)$ is a closed disc of radius $\delta$ with respect to a cone metric $g_{\co,0} = dr^2 + r^2 \cdot g_L$. Standard arguments from Riemannian geometry give the diameter of $D_0(\delta)$ to be $2 \delta$. We can then take $\delta < \frac{\epsilon}{2}$ to satisfy the first condition.
    
    Next, we consider the second condition, and note $\pi^{-1}(D_0(\delta)) = \what{T}(\delta)$. Using \eqref{eqn-diam-a-bound}, we see that the result follows once we choose $\delta$ small enough such that $\delta < C^{-1} \epsilon$, $\delta < \frac{\epsilon}{2}$, and $a_0 = \frac{\delta}{K}$.

    
\end{proof}

We can now apply Lemma \ref{lem-SM-bigbox} to prove convergence of the three classes of metrics on the small resolution:

\begin{itemize}
    \item Theorem \ref{thm-CdlO-SR-GH} - Convergence of the local models $(\what{T}(1),\what{d}_{\co,a}) \rightarrow (D_0(1), d_{\co,0})$:

    In this case, we have only one ODP singularity $s$. By the diameter estimate of $g_{\co,a}$ (Lemma \ref{lem-suff-small-discs-SR-model}), we see that for each $\eps > 0$, we can pick the set $\overline{G} = D_0(\delta)$ for an appropriately small $\delta > 0$ such that Lemma \ref{lem-SM-bigbox} applies.

    \item Theorem \ref{thm-FLY-SR-GH} - Convergence of the global balanced metrics $(\what{X}, \what{d}_{\FLY,a}) \rightarrow (X_0, d_{\FLY,0})$: 

    Here we use the fact that the Fu--Li--Yau metrics are, up to scaling, just the Candelas--de la Ossa metrics in a compact set around the $(-1,-1)$-curves $E_i$ and the ODP singularities $s_i$. For $\eps > 0$, we can pick $\overline{G_i} = D_0(\delta_i)$ for appropriately small $\delta_i$ around each singular point $s_i$.  Coupling this with the smooth convergence of the Fu--Li--Yau metrics on compact sets away from the $(-1,-1)$-curves, we may apply Lemma \ref{lem-SM-bigbox}.

    \item Theorem \ref{thm-HYM-SM-GH} - Convergence of the global HYM metrics $(\what{X},\what{d}_{\what{H}_a}) \rightarrow (X_0,d_{H_0})$:

    By Proposition \ref{propn-HYM-ests-SR}, we have the estimate 
    \begin{equation*}
        C^{-1} \cdot g_{\co,a} \leq \what{H}_a \leq C \cdot g_{\co,a}    
    \end{equation*}
    on the local sets $D_0(\delta_i)$ around each singularity $s_i$ where the Fu--Li--Yau metrics are a scaling of the Candelas--de la Ossa metrics. Lemma \ref{lem-suff-small-discs-SR-model} implies that for $\eps > 0$, there exists $\delta_i > 0$ and $a_0 > 0$ such that for all $0 < a < a_0$ then
    \begin{equation*}
        \mathrm{diam}_{H_0}(D_0(\delta_i)) < \epsilon, \quad \what{\diam}_{\what{H}_a}(\pi^{-1} (D_0(\delta_i))) < \epsilon.   
    \end{equation*}
We may therefore apply Lemma \ref{lem-SM-bigbox}.
    
\end{itemize}

\subsection{Estimates on the Smoothing}

\label{subsect-ests-SM}

We now prove the analogous statements on the smoothings. We will derive diameter bounds on $D_t(\beta_{t,\delta}) \subset V_t$ with respect to $g_{\co,t}$. Volume bounds can also be obtained in a similar way as for the small resolution, and we omit the details as they are not needed in the current work.

We recall that $\beta_{t,\rho}$ is defined by
\begin{equation*}
    \beta_{t,\rho} = \Big( \rho^3 + \frac{|t|^2}{4 \rho^3} \Big)^{\frac{1}{3}}
\end{equation*}
and the role of $\beta_{t,\rho}$ is so that the set $\{ r = \beta_{t,\rho} \} \subset V_t$ on the smoothing is identified with the set $\{ r = \rho \} \subset V_0$ on the cone via the map $\Phi_t$. 

The method we will use is analogous to that of the small resolutions. First, we use the Asymptotically Conical Decay Property \ref{prop-SM-II} to set $K > 0$ such that 
\begin{equation*}
    |(\Phi_t)^* (g_{\co,t}) - g_{\co,0}| \leq \frac{1}{2}
\end{equation*}
when $r > |t|^{\frac{1}{3}}K$. 

We will then split our region of interest $D_t(\beta_{t,\delta})$ into a ``disc" $D_t(\beta_{t,|t|^{\frac{1}{3}}K})$ and an ``annulus" $D_t(\beta_{t,\delta})\, \bs\, D_t(\beta_{t,|t|^{\frac{1}{3}}K})$.

\subsubsection{Bounds on the Disc}
\label{subsubsect-disc-SM}
We start by estimating the geometry of the disc $D_t(\beta_{t, |t|^{1/3}K})$. For this, we note that
\begin{equation*}
  S_{t^{-\frac{1}{3}}}:  \bigg( D_t(\beta_{t,|t|^{\frac{1}{3}}K}),g_{\co,t} \bigg) \rightarrow \bigg(D_1(\beta_{1,K}), |t|^{\frac{2}{3}} g_{\co,1} \bigg)
\end{equation*}
is an isometry. This is due to the scaling property $g_{\co,t} = |t|^{2/3} S^*_{t^{-1/3}}(g_{\co,1})$. It follows that
\begin{equation}
\label{eqn-diam-sphere}
    \diam_{\co,t} (D_t(\beta_{t,|t|^{\frac{1}{3}}K})) = |t|^{\frac{1}{3}} \cdot \diam_{\co,1} (D_1(\beta_{1,K})).
\end{equation}

\subsubsection{Annular Bounds}
\label{subsubsect-annular-SM}

Let $\delta > 0$. We now compute diameter bounds on the ``annular" region \\ $D_t(\beta_{t,\delta}) \, \bs\, D_t(\beta_{t,|t|^{\frac{1}{3}}K})$ when $0 < |t| \leq (\frac{\delta}{K})^3$. As before, this relies on \ref{prop-SM-II}.

Let $q \in D_t(\beta_{t,\delta}) \, \bs\, D_t(\beta_{t, |t|^{\frac{1}{3}}K})$ be an arbitrary point in the annular region. We will construct a curve $\tilde{\gamma}$ from $D_t(\beta_{t, |t|^{\frac{1}{3}}K})$ to $q$ and estimate its length $L_{\co,t}(\wtilde{\gamma})$. To do this, we will bring the setup back to the cone $V_0$ and use a radial ray.

Since $\Phi_t$ is a diffeomorphism on the annular region, we can write $q = \Phi_t(p)$ for $p \in V_0$. We note that $\beta$ is defined such that $|t|^{\frac{1}{3}}K < r(p) \leq \delta$ and we define $\rho>0$ by $r(p) = \rho$. Hence $|t|^{1/3} K < \rho \leq \delta$.

We can define a path $\gamma \: [|t|^{\frac{1}{3}} \frac{K}{\rho}, 1] \rarr V_0$ by
\begin{equation}
\label{eqn-gamma}
    \gamma(s) = s^{\frac{3}{2}} \cdot p.
\end{equation}
This path is chosen such that it begins in $D_0(|t|^{\frac{1}{3}}K)$ and moves outward along a ray emanating from $0$ to reach $\gamma(1) = p$. It can be checked that
\begin{equation*}
    \dot{\gamma}(s) = \rho \cdot \frac{\del}{\del r} \text{ and } r(\gamma(s)) = \rho \cdot s.
\end{equation*}
The corresponding path in $V_t$ is $\tilde{\gamma} = \Phi_t \circ \gamma$, and our goal is to estimate its length. We start with
\begin{equation*}
\begin{aligned}
    |L_{\co,t}(\wtilde{\gamma}) - L_{\co,0}(\gamma)| &\leq \int_{|t|^{1/3} \frac{K}{\rho}}^1 \Bigg| |\dot{\gamma}|_{(\Phi_t)^* (g_{\co,t})} - |\dot{\gamma}|_{g_{\co,0}} \Bigg| \, ds \\
    &\leq \int_{|t|^{1/3} \frac{K}{\rho}}^1  |(\Phi_t)^* g_{\co,t} - g_{\co,0}|_{g_{\co,0}} |\dot{\gamma}|_{g_{\co,0}} \, ds.
\end{aligned}
\end{equation*}
Using $|\dot{\gamma}|_{g_{\co,0}} = \rho$, $r(s) = s \cdot \rho$, and \ref{prop-SR-II}, we get
\begin{equation*}
    |L_{\co,t}(\wtilde{\gamma}) - L_{\co,0}(\gamma)| \leq C \int_{|t|^{1/3} \frac{K}{\rho}}^1  |t| r^{-3}(s) \rho \, ds \leq C \bigg( \frac{|t|^{1/3}}{K^2} - \frac{|t|}{\rho^2} \bigg).
\end{equation*}
We can also check that
\begin{equation*}
    L_{\co,0} (\gamma) = \int_{|t|^{1/3} \frac{K}{\rho}}^1 \rho \, ds =  \rho - |t|^{1/3} K.
\end{equation*}
By the triangle inequality and the above estimates, we conclude
\begin{equation*}
    L_{\co,t}(\wtilde{\gamma}) \leq C( \rho + |t|^{1/3} K^{-2}).
\end{equation*}
Since $|t| \leq \delta^3 K^{-3}$ and $\rho \leq \delta$, we conclude that
\begin{equation} \label{eqn-d_{co,t}}
    d_{\co,t} \Big(q, D_t(\beta_{t,|t|^{\frac{1}{3}}K}) \Big) \leq C \delta.
\end{equation}
Combining this with our diameter bound \eqref{eqn-diam-sphere} for $D_t(\beta_{t,|t|^{\frac{1}{3}}K})$, we get
\begin{equation} \label{eqn-diam-t-bound}
 \diam_{\co,t} \Big( D_t(\beta_{t,\delta}) \Big) \leq C \delta, \quad 0 < |t| \leq \frac{\delta^3}{K^3},
\end{equation}
which is the desired diameter bound for the Calabi-Yau metrics on the local model $(V_t,g_{\co,t})$.

\subsubsection{Bounds for the Fu--Li--Yau Metrics}
\label{subsubsect-FLY-bounds}

On the smoothings, the Fu--Li--Yau metrics are only close to scaled Candelas--de la Ossa metrics instead of being exactly equal to them. Due to this, we require a version of the diameter bound \eqref{eqn-diam-t-bound} for the Fu--Li--Yau metric. This will follow by virtue of the estimate \eqref{eqn-FLY-t-est}.

Consider a curve $\gamma$ on the disc $D_t(\beta_{t,\delta})$. We compare the length of this path $\gamma$ with respect to the Fu--Li--Yau metric and to a scaled Candelas--de la Ossa metric.
\begin{equation*}
    |L_{\FLY,t}(\gamma) - \sqrt{c} \cdot L_{\co,t}(\gamma)| \leq \frac{1}{\sqrt{c}} \int_{0}^1 |g_{\FLY,t} - c \cdot g_{\co,t}|_{g_{\co,t}} \cdot |\dot{\gamma}|_{g_{\co,t}} \, ds.
\end{equation*}
Using \eqref{eqn-FLY-t-est}, and recognizing that the $0<\delta<1$ and $K \gg 1$ appearing here can be chosen such that $\beta_{t,\delta}$ is smaller than the $\delta$ appearing in \eqref{eqn-FLY-t-est} for all $0 < |t| \leq  (\frac{\delta}{K})^3$, we have
\begin{equation*}
    |L_{\FLY,t}(\gamma) - \sqrt{c} \cdot L_{\co,t}(\gamma)| \leq C |t|^{\frac{2}{3}} \, L_{\co,t}(\gamma).
\end{equation*}
Thus
\begin{equation*}
    L_{\FLY,t}(\gamma) \leq C(|t|^{\frac{2}{3}}+1) \, L_{\co,t}(\gamma).
\end{equation*}
It then follows that
\begin{equation*}
    \diam_{\FLY,t} (D_t(\beta_{t,\delta})) \leq C (|t|^{\frac{2}{3}} +1) \diam_{\co,t} \Big( D_t(\beta_{t,\delta}) \Big).
\end{equation*}
Combining this with \eqref{eqn-diam-t-bound}, we have
\begin{lem}
\label{lem-diam-FLY-t-bounds}
    For $\delta > 0$, we have
    \begin{equation}
    \label{eqn-diam-FLY-t-bound}
        \diam_{\FLY,t} (D_t(\beta_{t,\delta})) \leq C \delta (|t|^{\frac{2}{3}}+1),
    \end{equation}
    for all $0 < |t| \leq (\frac{\delta}{K})^3$.
\end{lem}

\subsubsection{Applying the Main Lemma}
Using the diameter estimates, we prove an analog of Lemma \ref{lem-suff-small-discs-SR-model} in the case of the smoothings for the Candelas--de la Ossa and Fu--Li--Yau metrics.

\begin{lem}
\label{lem-suff-small-discs-SM-model}
    For $0 < \eps < 1$, there exists $\delta > 0$ and $0 < t_0$ such that for $0 < |t| < t_0$,
    
    \begin{enumerate}[label = \roman*)]
        \item $\diam_{\co,0} (D_0(\delta)) < \eps$, and
        
        \item $\diam_{\co,t} (D_t(\beta_{t,\delta})) < \eps$.
    \end{enumerate}

    The result also holds when taking diameters with respect to the Fu--Li--Yau metrics $g_{\FLY,0}$ and $g_{\FLY,t}$ instead of the Candelas--de la Ossa metrics $g_{\co,0}$ and $g_{\co,t}$.
\end{lem}



\begin{proof}
    As was the case in Lemma \ref{lem-suff-small-discs-SR-model}, the first condition holds as long as $\delta < \frac{\eps}{2}$. Using \eqref{eqn-diam-t-bound} or \eqref{eqn-diam-FLY-t-bound}, we see that for all $0 < |t| \leq (\frac{\delta}{K})^3$, we have the estimate
    \begin{equation}
        \diam_{\co,t} \Big( D_t(\beta_{t,\delta})\Big) \leq C \delta
    \end{equation}
    for some uniform constant $C > 0$. The result follows.
\end{proof}

Applying our lemma then gives convergence of our metrics on the smoothings:

\begin{itemize}
    \item Convergence of the local models $(D_t(\beta_{t,1}), d_{\co,t}) \rarr (D_0(1), d_{\co,0})$:

    We have only one ODP singularity $s$ for this case. Using our diameter estimate \eqref{eqn-diam-t-bound}, we have that for $\eps > 0$, we can pick $\overline{G} = D_0(\delta)$, $K_t = D_0((\frac{|t|}{2})^{\frac{1}{3}})$, and $C_t = D_t(|t|^{\frac{1}{3}})$ for sufficiently small $\delta > 0$ and $t$ such that Lemma \ref{lem-SM-bigbox} applies with the maps $\Phi_t$.

    \item Convergence of the global balanced metrics $(X_t,d_{\FLY,t}) \rarr (X_0, d_{\FLY,0})$:

    Here, we use the diameter estimate \eqref{eqn-diam-FLY-t-bound} instead. For $\eps > 0$, we can again pick $\overline{G_i} = D_0(\delta_i)$, $K_{i,t} = D_0((\frac{|t|}{2})^{\frac{1}{3}})$, and $C_{i,t} = D_t(|t|^{\frac{1}{3}})$ for sufficiently small $\delta_i > 0$ and $t$ around each singularity $s_i$. As such, we can apply the lemma with the maps $\Phi_t$.

    \item Convergence of the global HYM metrics $(X_t,d_{H_t}) \rarr (X_0,d_{H_0})$:

    Similarly to the case of the small resolutions, we use estimate \eqref{hym-sm-est}, which is,
    \begin{equation}
        C^{-1} \cdot g_{\FLY,t} \leq H_t \leq C \cdot g_{{\rm FLY,t}}
    \end{equation}
    on the local sets $D_0(\delta_i)$ around the singularities $s_i$. Set $K_{i,t} = D_0((\frac{|t|}{2})^{\frac{1}{3}})$ and $C_{i,t} = D_t(|t|^{\frac{1}{3}})$. Lemma \ref{lem-suff-small-discs-SM-model} then gives that for $\eps > 0$, there exists $\delta_i > 0$ and $t_0 > 0$ such that for all $0 < |t| < t_0$,
    \begin{equation}
        \diam_{H_0} (D_0(\delta_i)) < \eps, \quad \diam_{H_t} (D_t(\beta_{t,\delta_i})) < \eps.
    \end{equation}
    Applying Lemma \ref{lem-SM-bigbox} using the maps $\Phi_t$ proves the result.
\end{itemize}

Combining these results with the analogous results for the small resolution found at the end of section~\ref{subsubsect-annular-SR}, we obtain Theorem~\ref{thm-GH-sing}.

\begin{appendix}

\section{Hermitian--Yang--Mills Metrics on the Resolution} \label{appendix-hym}

The presentation in \cite{CPY21} only uses convergence of Hermitian--Yang--Mills metrics along a subsequence of the Fu--Li--Yau metrics $\what{\omega}_{\FLY,a_k}$ as $a_k \rightarrow 0$. Convergence along the full sequence $a \rightarrow 0$ also follows from the estimates in \cite{CPY21}, and in this section we provide the full details.
\smallskip
\par We start by establishing notation. We denote the components of a Hermitian metric $H$ on $T^{1,0}X$ by $H_{\alpha \bar{\beta}}$, and this convention is such that the associated inner product on $T^{1,0}X$ is given by
\begin{equation*}
\langle u, v \rangle = u^\alpha H_{\alpha \bar{\beta}} \overline{v^\beta} \quad u, v \in \Gamma(T^{1,0}X).
\end{equation*}
The components of the inverse of $H$ are denoted $H^{\bar{\mu} \nu}$ so that $H_{\mu \bar{\nu}} H^{\bar{\nu} \kappa} = \delta_\mu{}^\kappa$. The Chern connection associated to $H$ will be denoted $\nabla$, so that
\begin{equation*}
\nabla_k u^\alpha = \partial_k u^\alpha + u^\beta (\partial_k H_{\beta \bar{\gamma}} H^{\bar{\gamma} \alpha}), \quad \nabla_{\bar{k}} u^\alpha = \partial_{\bar{k}} u^\alpha.    
\end{equation*}
The Chern curvature of $H$ will be denoted by $F \in \Lambda^{1,1}({\rm End} \, T^{1,0}X)$ with conventions
\begin{equation*}
   F_\beta{}^\alpha{}_{j \bar{k}} = -\partial_{\bar{k}} (\partial_j H_{\beta \bar{\gamma}} H^{\bar{\gamma} \alpha}). 
\end{equation*}
We often omit the endomorphism indices and write $F_{j \bar{k}} = - \partial_{\bar{k}}(\partial_j H H^{-1})$. Given two metrics $H$ and $\hat{H}$, the difference in curvature tensors is
\begin{equation} \label{diff-curv}
    (F_H)_{j \bar{k}} - (F_{\hat{H}})_{j \bar{k}} = - \partial_{\bar{k}} (\hat{\nabla}_j h h^{-1}), \quad h = H \hat{H}^{-1},
\end{equation}
where $\hat{\nabla} h = \partial h + [h,\partial \hat{H} \hat{H}^{-1}]$ is the induced connection on ${\rm End} \, T^{1,0}X$.

Let $\omega_a = i (g_a)_{j \bar{k}} dz^j \wedge d \bar{z}^k$ be the sequence of Fu--Li--Yau balanced metrics on the resolution $\hat{X}$. To ease notation, in this section we write $\omega_a$ instead of $\what{\omega}_{\FLY,a}$. We will use the notation
\begin{equation*}
    \sqrt{-1} \Lambda_\omega F = g^{j \bar{k}} F_{j \bar{k}}.
\end{equation*}
From \cite{CPY21}, there is a sequence $H_a$ of Hermitian--Yang--Mills metrics on $T^{1,0}X$ solving 
\begin{equation*}
    \sqrt{-1} \Lambda_{\omega_a} F_{H_a} = 0,
\end{equation*}
along with estimates
\begin{equation} \label{HYM-estimates}
    C^{-1} \cdot g_a \leq H_a \leq C \cdot g_a, \quad r |\nabla H_a|_{g_a} \leq C,
\end{equation}
where the constants are uniform in $a$. Furthermore, for each $K \subseteq \hat{X} \backslash E$ there are estimates
\begin{equation*}
    \sup_K |\nabla^j H_a|_{g_a} \leq C(K,j).
\end{equation*}
We first normalize the sequence $\{ H_a \}$. We define
\begin{equation*}
    f_a := \log \frac{\det H_a}{\det g_a},
\end{equation*}
and after replacing $H_a$ with $e^{-c_a} H_a$, we can fix the normalization
\begin{equation} \label{HYM-normalization}
    \int_{\hat{X}} f_a \, d {\rm vol}_{g_a} = 0.
\end{equation}
Since $C^{-1} g_a \leq H_a \leq C g_a$, the constants $e^{-c_a}$ are uniformly bounded, and so the normalized sequence $\{ H_a \}$ still satisfies the estimates \eqref{HYM-estimates}. By the Arzel\`{a}--Ascoli theorem, for each $K \subset \hat{X} \backslash E$ there is a subsequence $H_{b_i} \rightarrow H_0$ converging uniformly. Taking an exhaustion of compact sets and identifying $\hat{X} \backslash E$ with $(X_0)_{\rm reg}$, we obtain a subsequence $H_{b_i} \rightarrow H_0$ which converges pointwise on $(X_0)_{\rm reg}$ and uniformly on compact subsets.

\smallskip
\par The goal of this appendix is to upgrade the subsequential convergence to convergence of the full sequence $\{ H_a \}$ uniformly on compact subsets of $(X_0)_{\rm reg}$. 

\begin{lem} \label{lem:fullconv}
For any compact sets $K \subseteq \what{X} \backslash E$, we have convergence
\[
H_a \rightarrow H_0
\]
in the $C^\ell(K)$ norm for any integer $\ell$.
\end{lem}

We proceed by contradiction. Suppose not, so that there exists an $\epsilon>0$, a compact subset $K_0 \subseteq (X_0)_{\rm reg}$, and a subsequence $a_i \rightarrow 0$ with
\begin{equation} \label{HYM-subseq}
    \sup_{K_0} |H_{a_i} - H_0|_{g_0} \geq \epsilon
\end{equation}
for all $a_i$. We can apply the estimates \eqref{HYM-estimates} and the Arzel\`{a}--Ascoli theorem to the $\{H_{a_i} \}$ to extract a further subsequence $\{ H_{a_{i_k}} \}$ converging uniformly to a limit $\tilde{H}_0$ on compact subsets of $(X_0)_{\rm reg}$. 

We now have two limiting metrics $H_0$ and $\tilde{H}_0$. Our goal will be to show that $H_0=\tilde{H}_0$. Taking the limit of \eqref{HYM-subseq} along the subsequence $a_{i_k}$ gives
\begin{equation*}
    \sup_{K_0} |\tilde{H}_0 - H_0|_{g_0} \geq \epsilon,
\end{equation*}
which is the desired contradiction.

The main step in showing $H_0 = \tilde{H}_0$ is the following:

\begin{lem} \label{lem:uniqueness}
Equip $(X_0)_{\rm reg}$ with the Fu-Li-Yau metric $g_0$. Let $H_0$ and $\tilde{H}_0$ be two metrics on the holomorphic tangent bundle of $(X_0)_{\rm reg}$, each satisfying
\begin{equation} \label{HYM-limit-estimates}
    \omega_0^2 \wedge F_{H} = 0, \quad C^{-1} \cdot g_0 \leq H \leq C \cdot g_0, \quad r |\nabla H|_{g_0} \leq C.
\end{equation}
Then $H_0 = \lambda \tilde{H}_0$ for a constant $\lambda \in \mathbb{C}$.
\end{lem}

We will prove Lemma \ref{lem:uniqueness} below, but let us assume it for now and complete the argument. Combining Lemma \ref{lem:uniqueness} with the normalization condition \eqref{HYM-limit-int} proved below, we obtain the uniqueness $H_0 = \tilde{H}_0$. This completes the proof of Lemma \ref{lem:fullconv}.

The normalization \eqref{HYM-normalization} along the sequence leads to the following normalization for the limit:

\begin{lem} Set $e^u = \tilde{H}_0 H_0^{-1}$. Then
\begin{equation} \label{HYM-limit-int}
    \int_{(X_0)_{\rm reg}} ({\rm Tr} \, u) \, d {\rm vol}_{g_0} = 0.
\end{equation}
\end{lem}
\begin{proof}
We have two subsequences $\{ b_i \}$ and $\{ a_i \}$ such that $f_{b_i} \rightarrow f_0$ and $f_{a_i} \rightarrow \tilde{f}_0$ pointwise on $\hat{X} \backslash E$ and uniformly on compact sets. Taking the logarithm of
\begin{equation*}
    \frac{\det \tilde{H}_0}{\det g_0} = (\det e^u)  \frac{\det H_0}{\det g_0}
\end{equation*}
we have $\tilde{f}_0 = {\rm Tr} \, u + f_0$. Thus if we can show that
\begin{equation*}
    \int_{(X_0)_{\rm reg}} f_0 \, d {\rm vol}_{g_0} = 0, \quad  \int_{(X_0)_{\rm reg}} \tilde{f}_0 \, d {\rm vol}_{g_0} = 0
\end{equation*}
we will have established \eqref{HYM-limit-int}. The calculation of both these integrals is the same, so we only calculate for the sequence $\{ b_i \}$. We split the integral as
\begin{equation*}
    \int_{(X_0)_{\rm reg}} f_0 \, d {\rm vol}_{g_0} = \int_{\{r < \delta \}} f_0 \, d {\rm vol}_{g_0}+ \int_{ \{ r \geq \delta \}} f_0 \, d {\rm vol}_{g_0}
\end{equation*}
The first integral can be estimated by using $|f_0| \leq C$ \eqref{HYM-limit-estimates}, so that
\begin{align*}
    \bigg| \int_{\{r < \delta \}} f_0 \, d {\rm vol}_{g_0} \bigg| &\leq C \int_{\{r < \delta \}} d {\rm vol}_{g_0} \leq C \delta^6.
\end{align*}
The second integral can be estimated by passing the limit under the integral over compact sets and using $\int_X f_i \, d {\rm vol}_{g_i} = 0$.
\begin{align*}
    \bigg| \int_{\{r \geq \delta \}} f_0 \, d {\rm vol}_{g_0} \bigg| &=  \lim_{b_i \rightarrow 0} \bigg| \int_{\{r \geq \delta \}} f_{b_i} d {\rm vol}_{g_{b_i}} \bigg|\nonumber\\
    &=  \lim_{b_i \rightarrow 0}\bigg| \int_{\{r < \delta \}} f_{b_i} d {\rm vol}_{g_{b_i}} \bigg| \nonumber\\
    &\leq  C \limsup_{b_i \rightarrow 0} {\rm Vol}_{\co,b_i} (\hat{T}(\delta))
\end{align*}
Here we used $C^{-1} g_b \leq H_b \leq C g_b$ implies $|f_b|\leq C$. We use the volume estimate \eqref{eqn-vol-a-bound} to conclude
\begin{equation*}
 \limsup_{b_i \rightarrow 0}   {\rm Vol}_{\co,b_i} (\hat{T}(\delta)) \leq C \delta^6.
\end{equation*}
We can now send $\delta \rightarrow 0$ to complete the proof of the lemma.
\end{proof}

All that remains now is to prove  Lemma \ref{lem:uniqueness}. We need to show that $h  = \tilde{H}_0 H_0^{-1}$ is a multiple of the identity endomorphism. Taking the trace of \eqref{diff-curv} and using the Hermitian--Yang--Mills equation gives the following equation for $h$.
\begin{equation} \label{perturbed-HYM}
 (g_0)^{j \bar{k}} \partial_{\bar{k}} ((\nabla^{H_0})_j h h^{-1}) = 0.
\end{equation}
The following key identity was observed by Uhlenbeck--Yau \cite{uhlenbeckyau86}. We will use a version from Jacob--Walpuski \cite{jacobwalpuski} and give the proof for completeness.

\begin{lem}
Let $H= e^u \hat{H}$ be a pair of Hermitian metrics on a holomorphic bundle $E \rightarrow X$ over a Hermitian manifold $(X,g)$. Write $h = e^u$. Then we have the identity
    \begin{equation} \label{Laplace-u-id}
        \Delta_g |u|^2_{\hat{H}} = 2 \langle g^{\bar{k} j} \partial_{\bar{k}} (\hat{\nabla}_j h h^{-1}), u \rangle_{\hat{H}} + 2 g^{\bar{k} j} \langle \hat{\nabla}_j h h^{-1}, \hat{\nabla}_k u \rangle_{\hat{H}}.
    \end{equation}
If we assume $|u|_{\hat{H}} \leq R$ then there exists a constant $C(R)>0$ such that the following estimate holds:
        \begin{equation} 
        \label{Laplace-u-est}
        \Delta_g |u|^2_{\hat{H}} \geq 2 \langle g^{j \bar{k}} \partial_{\bar{k}} (\hat{\nabla}_j h h^{-1}), u \rangle_{\hat{H}} + \frac{1}{C} |\hat{\nabla} u|^2_{\hat{H}}.
    \end{equation}
 
\end{lem}

\begin{proof}
We start by recalling the definition of $u$. Let $\{ e_a \}$ be a local smooth frame for $T^{1,0}X$ such that
\begin{equation*}
\hat{H} = \sum_{a=1}^3 e^a \otimes \overline{e^a}, \quad h = \sum_{a=1}^3 \lambda_a \, e_a \otimes e^a.
\end{equation*}
Then $u$ is defined by
\begin{equation*}
    u = \sum_{a=1}^3 (\log \lambda_a) \, e_a \otimes e^a.
\end{equation*}
In this frame, the adjoint $\dagger$ with respect to $\hat{H}$ is just the conjugate-transpose of the components, and so $h^\dagger = h$ and $u^\dagger =u$.

The connection coefficients are $\hat{\nabla}_i e_a = A_{i a}{}^b e_b$ and metric compatibility implies $A_{i a}{}^b = -\overline{A_{\bar{i} b}{}^a}$. We can then compute
\begin{equation*}
(\hat{\nabla}_i u)_a{}^b = \frac{\partial_i \lambda_a}{\lambda_a} \delta_a{}^b + (\log \lambda_a - \log \lambda_b) A_{ia}{}^b
\end{equation*}
and
\begin{equation*}
(\hat{\nabla}_i h h^{-1})_a{}^b = \frac{\partial_i \lambda_a}{\lambda_a} \delta_a{}^b + \lambda_b^{-1}(\lambda_a-\lambda_b) A_{ia}{}^b.
\end{equation*}
The inner product on endomorphisms is $\langle u,v \rangle_{\hat{H}} = {\rm Tr} \, u v^\dagger_{\hat{H}}$, and to ease notation we drop the subscript $\hat{H}$. We have
\begin{equation*}
    \partial_i |u|^2 = \langle \hat{\nabla}_i u,u \rangle + \langle u, \hat{\nabla}_{\bar{i}} u \rangle = 2 \langle \hat{\nabla}_i u, u \rangle,
\end{equation*}
since $u^\dagger = u$ and $(\nabla_{\bar{i}} u)^\dagger = \hat{\nabla}_i u$. Furthermore, we notice that the expressions above imply
\begin{equation*}
    \langle \hat{\nabla}_i u, u \rangle = \langle \hat{\nabla}_i h h^{-1}, u \rangle,
\end{equation*}
since the inner product only picks up the diagonal part. Therefore
\begin{equation*}
\begin{aligned}
g^{\bar{k} j} \partial_{\bar{k}} \partial_j |u|^2 &= 2 g^{\bar{k} j} \partial_{\bar{k}} \langle \hat{\nabla}_j h h^{-1}, u \rangle \nonumber\\
&= 2 \langle g^{\bar{k} j} \partial_{\bar{k}} (\hat{\nabla}_j h h^{-1}), u \rangle + 2 g^{\bar{k} j} \langle \hat{\nabla}_j h h^{-1}, \hat{\nabla}_k u \rangle.
\end{aligned}
\end{equation*}
This proves identity \eqref{Laplace-u-id}. For the estimate, it remains to show
\begin{equation} \label{mixed2log-ineq}
    g^{\bar{k} j} \langle \hat{\nabla}_j h h^{-1}, \hat{\nabla}_k u \rangle \geq \frac{1}{C} g^{\bar{k} j} \langle \hat{\nabla}_j u, \hat{\nabla}_k u \rangle.
\end{equation}
We compute
\begin{equation*}
 \langle \hat{\nabla}_j h h^{-1}, \hat{\nabla}_k u \rangle = (\hat{\nabla}_j h h^{-1})_a{}^b (\hat{\nabla}_{\bar{k}} u)_b{}^a
\end{equation*}
which gives
\begin{equation*}
     \langle \hat{\nabla}_j h h^{-1}, \hat{\nabla}_k u \rangle = \sum_{a,b} \partial_j \log \lambda_a \partial_{\bar{k}} \log \lambda_a + \log \frac{\lambda_a}{\lambda_b} \lambda_a^{-1} (\lambda_b-\lambda_a) A_{ja}{}^b A_{\bar{k} b}{}^a
\end{equation*}
and so
\begin{equation*}
    g^{\bar{k} j} \langle \hat{\nabla}_j h h^{-1}, \hat{\nabla}_k u \rangle = \sum_a |\partial \log \lambda_a|^2 + \sum_{a,b} \log \frac{\lambda_a}{\lambda_b} \lambda_a^{-1} (\lambda_a-\lambda_b) |A_a{}^b|^2.
\end{equation*}
On the other hand,
\begin{equation*}
 g^{\bar{k} j} \langle \hat{\nabla}_j u, \hat{\nabla}_k u \rangle =    \sum_a |\partial \log \lambda_a|^2 + \sum_{a,b} \bigg|\log \frac{\lambda_a}{\lambda_b} \bigg|^2 |A_a{}^b|^2
\end{equation*}
To show \eqref{mixed2log-ineq}, it suffices to prove
\begin{equation*}
 \log \frac{\lambda_a}{\lambda_b} \lambda_a^{-1} (\lambda_a-\lambda_b) \geq \frac{1}{C} \bigg| \log \frac{\lambda_a}{\lambda_b} \bigg|^2.   
\end{equation*}
Let $e^x = \lambda_a/\lambda_b$, and recall that by assumption there holds $|x| \leq R$. We are seeking an inequality of the form
\begin{equation*}
x(1-e^{-x}) \geq \frac{1}{C} x^2, \quad |x| \leq R.
\end{equation*}
This inequality indeed holds since $x^{-1}(1-e^{-x}) > 0$ for all $x \in \mathbb{R}$, so a constant $C>1$ exists such that $x^{-1}(1-e^{-x}) \geq C^{-1}$ on the compact region $|x| \leq R$. 
\end{proof}

We now apply this lemma and combine it with \eqref{perturbed-HYM} to obtain
\begin{equation} \label{HYM-int-est}
    \Delta_{g_0} |u|^2_{H_0} \geq C^{-1} |\nabla u|^2_{g_0,H_0}
\end{equation}
on $(X_0)_{\rm reg}$. Let $\eta: [0,\infty) \rightarrow [0,\infty)$ be a cutoff function such that $\eta \equiv 0$ on $0 \leq r \leq 1$ and $\eta \equiv 1$ on $r \geq 2$. Let $\eta_\delta(r) = \eta(r/\delta)$, so that $\eta_\delta$ vanishes on $\{ r < \delta \}$. Identifying $(X_0)_{\rm reg}$ with $\hat{X} \backslash E$, we have
\begin{equation*}
\int_{\hat{X}} \eta_\delta |\nabla u|^2  \, d {\rm vol}_{g_0} \leq C \int_{\hat{X}} \eta_\delta \Delta |u|^2 \, d {\rm vol}_{g_0} .
\end{equation*}
Integrating by parts, we can estimate
\begin{align*}
\int_{\hat{X}} \eta_\delta \Delta |u|^2 \, d {\rm vol}_{g_0} &\leq \int_{\hat{X}} |\nabla \eta_\delta| |\nabla u| |u| \, d {\rm vol}_{g_0} \nonumber\\
&\leq C \delta^{-1} \int_{ \{ \delta < r < 2 \delta \} } r^{-1} d {\rm vol}_{g_0} \nonumber\\
&\leq C \delta^4,
    \end{align*}
by using $|\nabla u| \leq C r^{-1}$ \eqref{HYM-limit-estimates} and $|\nabla \eta_\delta| \leq C \delta^{-1}$. Sending $\delta \rightarrow 0$, we conclude 
\begin{equation*}
\int_{\hat{X} \backslash E} |\nabla u|^2 d {\rm vol}_{g_0} = 0.    
\end{equation*}
\par We conclude that $|\nabla u| = 0$ on $\hat{X} \backslash E$, and since $\nabla$ is the Chern connection this implies
\begin{equation*}
    \bar{\partial} u = 0
\end{equation*}
on $\hat{X} \backslash E$. Since $u$ is bounded, by Hartogs' theorem we may extend $u$ to all of $\hat{X}$. Thus $u$ is a holomorphic endomorphism of the tangent bundle $T^{1,0} \hat{X}$. Since $\hat{X}$ is a Calabi--Yau threefold with finite fundamental group, we conclude (see \eqref{no-endos}) that $u$ must be a multiple of the identity: $u = \lambda \, {\rm id}$. This completes the proof of Lemma \ref{lem:uniqueness}.

\section{The (Local) Map \texorpdfstring{$\Phi$}{Phi}}
\label{app-Phi}

We collect several results regarding the map $\Phi: V_0 \bs \{ 0 \} \rightarrow V_1$. We recall that $\Phi$ was defined by
\begin{equation*}
    \Phi(z) = z + \frac{\overline{z}}{2 \|z\|^2}.
\end{equation*}
We will show that $\Phi$ is a diffeomorphism from $\{z \in V_0 \mid \|z\|^2 > 1/2 \}$ to $\{z \in V_1 \mid \| z \|^2 > 1\}$.

We start by taking the norm:
\begin{equation*}
    \|\Phi(z)\|^2 = \|z\|^2 \cdot \Big( 1 + \frac{1}{4 \|z\|^4} \Big).
\end{equation*}
Let $x = \| z \|^2$, and remark that the function
\begin{equation*}
    f(x) = x \Big( 1 + \frac{1}{4x^2} \Big), \qquad x > 0
\end{equation*}
is strictly increasing on $(\frac{1}{2},\infty)$. From this, we see that $\Phi$ is injective on $V_0 \bs \{ \| z \|^2 \leq \frac{1}{2} \}$. Indeed, suppose that $\Phi(z) = \Phi(z')$ with $\| z \|^2> {\frac{1}{2}}$ and $\| z' \|^2 > {\frac{1}{2}}$. Then the restriction of the domain implies that $\|z\|^2 = \|z'\|^2$. From here, we can split the equation $\Phi(z) = \Phi(z')$ into real and imaginary parts and a straightforward calculation shows $z=z'$.

Our next step is to find an inverse for $\Phi$. First, we note that 
\begin{equation*}
    g(x) = \frac{1}{2} (x + \sqrt{x^2 - 1}), \qquad x > 1
\end{equation*}
is an inverse of $f: (\frac{1}{2},\infty) \rightarrow (1,\infty)$.

Now, let $w \in V_1$, with $\|w\|^2 = B > 1$. It can be checked by direct calculation that
\begin{equation*}
    z_i = \Big( \frac{2g(B)}{2g(B)+1} \Big) \cdot \Re(w_i) - \sqrt{-1} \cdot \Big( \frac{2g(B)}{2g(B)-1} \Big) \cdot \Im(w_i),
\end{equation*}
satisfies: 
\par $\bullet$ $z \in V_0$,
\par $\bullet$ $\|z\|^2 = g(B) > \frac{1}{2}$,
\par $\bullet$ $\Phi(z) = w$. 

It follows that $\Phi$ is a bijection from $\{z \in V_0 \mid \|z\|^2 > \frac{1}{2}\}$ to $\{z \in V_1 \mid \|z\|^2 > 1 \}$. The coordinate expressions for $\Phi$ and the ones appearing in the previous computations show that $\Phi$ and its inverse are smooth, hence $\Phi$ is a diffeomorphism between these sets. Written in terms of $r(z) = \|z\|^{\frac{2}{3}}$, we get:

\begin{propn}
\label{propn-Phi-diffeo}
The map $\Phi \: V_0 \bs \{z \in V_0 \mid r(z) > 2^{-\frac{1}{3}}\} \rarr \{z \in V_1 \mid r(z) > 1 \}$ is a diffeomorphism.
\end{propn}

We recall that we can compose the map $\Phi$ with scaling maps $S_R^*$. In particular, we defined $\Phi_t = S_{t^{\frac{1}{3}}} \circ \Phi \circ S_{t^{-\frac{1}{3}}}$. One can check that in coordinates, this map takes the form
\begin{equation*}
    \Phi_t(z) = z + \frac{t\overline{z}}{2 \|z\|^2}.
\end{equation*}

Smoothness of the scaling maps gives us the following corollary:

\begin{cor}
\label{cor-Phi_t-diffeo}
The map $\Phi_t \: V_0 \bs \{z \in V_0 \mid r(z) > 2^{-\frac{1}{3}} \cdot t^{\frac{1}{3}} \} \rarr \{z \in V_1 \mid r(z) > t^{\frac{1}{3}} \}$ are diffeomorphisms for $t > 0$.
\end{cor}

\end{appendix}

\bibliographystyle{amsalpha}
\cleardoublepage

\phantomsection

\renewcommand*{\bibname}{References}

\bibliography{biblio}

\end{document}